\documentclass[11pt,colorlinks,citecolor=blue,urlcolor=blue]{article}
\usepackage{mathrsfs}
\usepackage{amsmath, amsthm, amssymb}
\usepackage{mathtools}
\usepackage{amsfonts}
\usepackage[T1]{fontenc} 

\usepackage{palatino}
\usepackage{graphicx}
\usepackage{array}
\usepackage{gensymb}
\usepackage{multirow}
\usepackage{lscape}
\usepackage{geometry}
\usepackage{hyperref}
\usepackage{soul,color}
\usepackage{enumerate}
\usepackage{listings}
\usepackage{float}
\usepackage{cancel}

\usepackage[stretch=10]{microtype}
\newtheorem{theorem}{Theorem}[section]
\newtheorem{lemma}[theorem]{Lemma}
\newtheorem{proposition}[theorem]{Proposition}
\newtheorem{corollary}[theorem]{Corollary}

\theoremstyle{definition}

\newtheorem{remark}[theorem]{Remark}

\numberwithin{equation}{section}

\usepackage{tikz}
\usetikzlibrary[patterns]
\usetikzlibrary{arrows,shapes,trees, positioning} 
\usetikzlibrary{decorations.markings}
\tikzstyle{ann} = [fill=white,font=\footnotesize,inner sep=1pt]

\newenvironment{pf*}[1]{\medskip \noindent {\em #1}}{\endproof
\medskip}

	\newcommand{\R}{{\mathbb R}}
	
	\newcommand{\N}{{\mathbb N}}
	
	\newcommand{\Prob}{{\mathbb P}}
	\newcommand{\E}{{\mathbb E}}

	\renewcommand{\P}{{\mathcal P}}
	\renewcommand{\L}{{\mathcal L}}

	\renewcommand{\o}{{\omega}}

	\newcommand{\wt}{\widetilde}
	\newcommand{\wh}{\widehat}

	\newcommand{\zz}[1]{\mathbb{#1}}

	\renewcommand{\a}[1]{\left\lvert {#1}\right\rvert}

\begin{document}

\title{Critical Branching Brownian Motion with Killing}
\author{Steven P. Lalley\thanks{University of Chicago} and Bowei
Zheng\thanks{University of Chicago}}


%

\maketitle

\begin{abstract}
We obtain sharp asymptotic estimates for hitting probabilities of a
critical branching Brownian motion in one dimension with killing at
$0$. We also obtain sharp asymptotic formulas for the tail
probabilities of the number of particles killed at $0$. In the special
case of double-or-nothing branching, we give exact formulas for both
the hitting probabilities, in terms of elliptic functions, and the
distribution  of the number of killed particles.
\end{abstract}

\section{Introduction}\label{sec:intro}

\emph{Branching Brownian motion} is a stochastic
particle system in which each individual particle moves along a
Brownian trajectory, and at a random, exponentially distributed time
independent of its motion is replaced by a random collection of
identical offspring particles. The motions, gestation times, and
offspring numbers of different particles are conditionally
independent, given the times and locations of their births. Thus,
conditional on the event that at time $t$ there are $Z_{t}$ particles
at locations $x_{1},x_{2},\dots ,x_{Z_{t}}$, the law of the post-$t$
evolution is identical to that of $Z_{t}$ mutually independent
branching Brownian motions started by individual particles at the
locations $x_{i}$. A formal construction of the process is outlined in
section \ref{sec:construction} below.

The process $Z_{t}$ that records the total
number of particles at time $t$ is a \emph{continuous-time
Galton-Watson process}: see  \cite{AN72}, ch.~2 for the basic theory
of these. Branching Brownian motion is said to be 
\emph{supercritical}, \emph{critical} or \emph{subcritical}
according as the mean of the offspring distribution is  greater, equal
or less than $1$. In the critical and subcritical cases, the
particle population eventually dies out, with probability one,
provided the population starts with only finitely many particles; in
the supercritical case, however, there is positive probability that
the population blows up, that is,
$Z_{t}\rightarrow \infty$ as $t \rightarrow \infty$. Thus, the
questions that are germane to the supercritical case are different
from those of interest in the critical case.

It has been known since the work of McKean \cite{M75} that
supercritical branching Brownian motion is intimately related to the
behavior of solutions to the \emph{Fisher-KPP equation}.  In
particular, this equation governs the cumulative distribution function
$u(t,x)=\Prob^0(M_t\leq x)$ of the position $M_{t}$ of rightmost
particle at time $t$. Using this fact, McKean gave a probabilistic
proof of the Kolmogorov, Petrovsky, and Piscounov \cite{KPP37}
theorem, which asserts that the solution of the KPP equation with
Heaviside initial data stabilizes as a traveling wave of velocity
$\sqrt{2}$. Subsequently, Bramson \cite{B78} used the connection with
supercritical branching Brownian motion to obtain sharp estimates for
the center of the wave, and Lalley and Sellke \cite{LS87} showed that
the limiting traveling wave $w(x)$ can be represented as a as a
mixture of extreme-value distributions.

When the branching mechanism is \emph{critical} or
\emph{sub-critical}, a more natural object of study is the random variable
\begin{equation*}
    M=\sup_{t>0}M_t,
\end{equation*}
the rightmost location ever reached by a particle of the branching
Brownian motion. Critical branching Brownian motion has been proposed
as a model for the spatial displacement of alleles without selective
advantage or disadvantage, and in this context the distribution of $M$
plays an important role (see, for example, \cite{CG76}, \cite{S76},
\cite{SF79} and references therein).  Sawyer and Fleischman
\cite{SF79} proved that if the offspring distribution has mean $1$,
positive variance $\sigma^2$ and finite third moment, then the tail of
the distribution of $M$ satisfies the power law
\begin{equation}
    \label{asym.M.cdf}
    \Prob (M\geq x)\sim \frac{6}{\sigma^2 x^2} \quad \text{as} \; x\to \infty.
\end{equation}

Modifications of branching Brownian motion and branching random walk
in which the laws of reproduction and/or particle motion depend on
particle location arise in various contexts. See, for instance, Lalley
and Sellke \cite{LS88}, {\cite{LS89}} and Berestycki \emph{et al.} \cite{BBS13}, in
which particle reproduction is allowed only in certain favored regions
of space; Kesten \cite{K78}, Aldous \cite{A}, Addario-Berry and
Broutin \cite{AB11}, A\"id\'ekon, Hu, and Zindy \cite{AHZ13} and Maillard
\cite{M13} where particles are killed upon entering the half-line
$(-\infty ,0]$, and Berestycki \emph{et al.} \cite{BBHM15}; and Lalley
and Sellke \cite{LS92} and Korostelev and Korosteleva \cite{KK03},
\cite{K04}, \cite{KK04}, where particles move according to
spatially-inhomogeneous diffusion laws. In the articles \cite{AB11},
\cite{AHZ13}, \cite{M13}, and \cite{BBHM15}, the branching law is
supercritical, but particle production is balanced by the killing in
$(-\infty ,0]$ so that $M_{n}/n \rightarrow 0$.

This paper will focus on the modification of {critical} branching
Brownian motion (that is, where the mean number of offspring at
reproduction events is 1) in which particles are killed upon reaching the
interval $(-\infty ,0]$. Clearly, the number $Z_{t}$ of particles
alive at time $t$ in this process is dominated by the corresponding
random variable for the critical branching Brownian motion with no
killing, and so $Z_{t}=0$ eventually, with probability
$1$. Furthermore, the distribution of the maximal particle location
$M$ is dominated by that of the maximal particle location in critical
branching Brownian motion with no killing, and so the results of
Sawyer and Fleischman \cite{SF79} imply that for any $\varepsilon >0$
and any initial particle location $y>0$,
\[
	\Prob^{y} \{M\geq x \} \leq \frac{(6+\varepsilon)}{\sigma^{2}x^{2}}
\]
for all sufficiently large $x$.

It is by no means evident, however, that the tail behavior should be
the same as for branching Brownian motion with no killing. In fact we
will prove that when the
branching process is initiated by a single particle at a location
$y>0$ near zero, the tail follows a power law with exponent $3$ rather
than $2$. In particular, we will prove in Theorem
\ref{thm.max.general} that for each fixed $y>0$,
\[
	\Prob^y (M>x)\sim \frac{C_3y}{x^3} \quad \text{as} \;\; x
	\rightarrow \infty.
\]
where $C_{3}>0$ is a constant depending on the offspring distribution
but not on $x$ or $y$. On the other hand, we will show that for
initial particle locations $y=sx$ whose distances from the killing
zone are proportional to the target $x$, the exponent of the power law
reverts to $2$; in particular, there exists a continuous function
$C_{4} (s)$ of $s\in (0,1)$ such that 
\[
	\Prob^{sx} (M>x)\sim \frac{C_4 (s)}{x^2} \quad \text{as} \;\; x
	\rightarrow \infty.
\]
Furthermore, we will show that in the \emph{Moranian} case, where the
offspring law is double-or-nothing, the tail probability $\Prob^y
(M>x)$ can be explicitly written as a Weierstrass
$\mathcal{P}-$function. All of these results will be deduced from an
analysis of a boundary value problem in the variable $y$
satisfied by the hitting probability $\Prob^y (M>x)$.

 Also of interest is the total number $N$ of particles
killed at $0$. For supercritical branching Brownian
motion with particle drift and killing at $0$, Maillard \cite{M13} and
Berestycki \emph{et al.} \cite{BBHM15} have, under various hypotheses
concerning the drift and the reproduction mechanism, obtained sharp
estimates for the  tail of the distribution of $N$. For critical
branching Brownian motion with killing,
T.~Y.~Lee \cite{L90-1} proved a conditional limit
theorem for the distribution of $N$ given that $N\geq 1$: in
particular, he showed that as the position $y$ of the initial particle
$\rightarrow \infty$, so that $\Prob ^{y} (N\geq 1)\rightarrow 0$, the
$\Prob^{y}-$ conditional distribution of $N/P^{y} (N\geq 1)$
converges to a  non-degenerate limit distribution. (See also \cite{L90-2}
for a time-dependent analogue.) 

We will study the distribution of $N$ for critical branching Brownian
motion with killing at $0$
under a \emph{fixed} $\Prob^y$. In section~\ref{sec:number-killed} we
will show that, for offspring distributions with mean $1$, positive finite
variance $\sigma^{2}$, and finite third moment,
\[
	\sum_{j=1}^{k} j\Prob^{y} (N\geq j)\sim Cy\sqrt{k} \quad
	\text{as} \; k \rightarrow \infty .
\]
Under certain additional hypotheses on the offspring distribution, we
will show that the distribution of $N$ obeys a power law with exponent
$3/2$, thus verifying a conjecture of Professor Jian Ding, and   in
addition, we will show that the $N$ obeys an asymptotic \emph{local}
limit theorem.  In particular, we will prove that
\begin{equation*}
     \Prob^y(N\geq k)\sim \frac{C_7y}{k^{3/2}},
\end{equation*}
and
\begin{equation*}
    \Prob^y(N= k)\sim \frac{C_8y}{k^{5/2}},
\end{equation*}
where $C_7, C_8>0$ are constants depending on the offspring
distribution. Finally, in the Moranian case, we will give in Theorem
\ref{thm.killed.exact.moranian} an explicit formula for the tail
distribution of $N$.

\section{Construction and Monotonicity
Properties}\label{sec:construction}

Branching Brownian motions with initial
particle locations at points $y\in \zz{R}_{+}$ can be constructed on
any probability space that supports countably many (i) independent
standard Wiener processes $W^{i}$; (ii) independent, identically distributed
unit exponential random variables $T_{i}$; and (iii) independent, identically
distributed random variables $L_{i}$ all distributed according to the
prescribed offspring distribution. We dub this construction the
\emph{discrete Brownian snake}, as it is the natural discrete analogue
of Le Gall's Brownian snake: see \cite{L99} for details.

The construction proceeds by using the random variables $\{L_{i}
\}_{i\geq 0}$ to construct a \emph{Galton-Watson tree}. This
construction is standard: see \cite{AN72}. If the offspring
distribution has mean $1$, as we shall assume throughout, then the
resulting Galton-Watson tree is almost surely a finite, rooted tree
with vertices arranged in \emph{generations}, beginning with the root
at generation $0$. To each vertex $v$ is attached one of the random
variables $L_{i}$, with $L_{0}$ attached to the root; for each vertex
$v$ the random variable $L_{i}$ determines the number of offspring
vertices. The random variables $L_{i}$ can be attached to vertices in
any number of different ways, the most common being the
\emph{breadth-first} rule, in which the values $L_{i}$ are read
successively from the stack generation-by-generation, left-to-right. 

Given the realization of the Galton-Watson tree, we attach  unit
exponential random variables $T_{i}$ and standard Wiener processes
$W^{i}$ to the \emph{edges} of the tree in such a way that the index
$i=i (v)$ matches the index of the random variable $L_{i}$ attached to the
\emph{lower} vertex $v$ of the edge (the incident vertex with higher
generation number). The random variable $T_{i}$ attached to an edge
determines the real time elapsed between reproduction events, and
the Wiener process $W^{i}$  determines the displacement of the
particle in real time from its position at the last reproduction
event. Thus, the  particles alive at (real) time $t$ are in one-to-one
correspondence with the vertices $v$ of the tree such that 
\[
	\sum_{w< v} T_{i (w)}<t\leq \sum_{w\leq v}T_{i (w)};
\]
here the symbols $<$ and $\leq$ indicate the ordering of vertices $w$
along the geodesic path in the tree from the root to $v$. The spatial
position of the particle represented by vertex $v$ at time $t$ is
\[
	y+\sum_{w<v} W^{i (w)} (T_{i (w)})+W^{v} (t-\sum _{w<v} T_{i
	(w)}).
\]
Observe that these rules yield a \emph{simultaneous} construction of
branching Brownian motions from all initial positions $y$. It is
evident from this construction that the distribution of the maximum
position $M$ attained by a point of the branching Brownian motion is
stochastically monotone in the initial position $y$.

Branching Brownian motion with killing at $0$, or more generally with
killing at any point $z\leq 0$, can be constructed using the same
marked tree as for branching Brownian motion with no killing. The rule
is simple: once a trajectory along an edge enters $(-\infty,z]$,  the
tree is pruned at that point. This leaves a subtree of the original
Galton-Watson tree in which certain edges (those corresponding to
particles that are killed at $0$) are cut.  The vertices of this
subtree represent particles of the branching Brownian motion with
killing at $z$. Thus, the set of particles alive
in the branching Brownian motion with killing is a subset of the set
of particles in the branching Brownian motion with no killing, which
we will henceforth refer to as the \emph{enveloping} branching
Brownian motion.

This construction makes it obvious that the distribution of $M$ is
dominated by that for branching Brownian motion with no killing at
$z$, and that if $z_{2}<z_{1}$ then the distribution of $M$ for
branching Brownian motion with killing at $z_{1}$ is stochastically
dominated by that for branching Brownian motion with killing at
$z_{2}$. Furthermore, the implied inequalities among the cumulative
distribution functions are \emph{strict}: for instance, if $w_{2} (x)$
and $w_{1} (x)$ are the tail distributions of $M$ for branching
Brownian motions with killing at $z_{2}<z_{1}\leq 0$, respectively,
when both are initiated by a single particle at $0$ (that is, $w_{i}
(x)$ is the probability that $M\geq x$) then
\begin{equation}\label{eqn.strict.tail.ineq}
	w_{1} (x)<w_{2} (x).
\end{equation}
To see this, observe that  there is positive
probability that a branch will be pruned when there is killing at
$z_{1}$ but not when the killing is at $z_{2}$, and that this branch
will extend in such a way that it gives rise to a particle that
reaches location $x$. Finally, branching Brownian motions with killing
at $z$ converge as $z \rightarrow -\infty$ to branching Brownian
motion with no killing. Thus, for any $x>0$,
\begin{equation}\label{eqn.no-kill-lim}
	\lim_{ z \rightarrow \infty} w_{z} (x)=w_{\infty} (x),
\end{equation}
where $w_{z} (x)$ is the  probability that $M\geq x$ for branching
Brownian motion with killing at $-z$ and $w_{\infty} (x)$ is the
corresponding probability for branching Brownian motion with no
killing (both with initial particles located at $0$).

It should be obvious that minor variations of the construction just
outlined can be used to build a variety of related processes. One that
will prove useful in certain of the arguments to follow is
\emph{branching Brownian motion with freezing}, in which particles
that reach a target point $0$ (or, more generally, a closed set $B$) are
frozen in place, ceasing all motion and reproduction thereafter, but
not dying. In a critical branching Brownian motion with freezing of
particles at location $0$, eventually all existing particles will be
frozen at $0$; moreover, the number $N_{t}$ of particles frozen at
time $t$ is the same as the number of particles killed at $0$ up to
time $t$ in the corresponding  branching Brownian motion with killing
at $0$. 

Henceforth, we shall assume that all branching Brownian motions are
\emph{critical} and that the offspring distribution has positive,
finite variance $\sigma^{2}$, and we shall denote by 
\begin{equation}\label{eqn.Psi}
	\Psi (z)=\sum_{k=0}^{\infty}\Prob (L=k) z^{k}
\end{equation}
the probability generating function of the offspring distribution.

\section{Product Martingales and  Differential
Equations}\label{sec:product-mgs}

The key to our analysis will be the fact that hitting probabilities
and related expectations for critical branching Brownian motion,
viewed as functions of the initial point $y$, are governed by a
nonlinear second-order differential equation. This differential
equation is well known, but since we will have occasion to consider
expectations of complex-valued random variables, we shall spell out the
boundary value problems in detail.

Say that a sequence
$f:\zz{Z}_{+}\rightarrow \zz{C}$ is \emph{multiplicative} if it is a
geometric sequence of the form $f (n)=z^{n}$ for some
$z\in \zz{C}$. For any $A\in (0,\infty ]$, let $\Prob^{y}$ be the law of
a branching Brownian motion with initial point $y\in [0,A]$ in which
particles are frozen upon reaching either $0$ or $A$.  For  $i=0$ and
$i=A$ define 
\begin{equation}\label{eqn.N0NA}
	N_{i}=\text{number of particles frozen at $i$.}
\end{equation}
Both $N_{0}$ and $N_{A}$ are almost surely finite, since only finitely
many particles are born in the  course of a critical branching
Brownian motion. Clearly, $N_{A}=0$ when $A=\infty$.

\begin{proposition} If $f,g:\zz{Z}_{+}\rightarrow \zz{C}$ are
bounded, multiplicative sequences, then the function 
    \label{prop.varphi.ode}
    $\varphi(y)=\E^y\left[ f(N_0)g(N_A) \right]$ satisfies the second
    order differential equation 
    \begin{equation}
        \frac{1}{2}\varphi''(y)=\varphi(y)-\Psi(\varphi(y)) \quad
	\text{for all} \;\; y\in (0,A).
        \label{eqn.varphi.ode}
    \end{equation}
\end{proposition}

In the special case where $A=\infty$ and $f (n)=\delta_{0} (n)$ this
was stated and proved by Sawyer and Fleischman \cite{SF79}, and this
proof was subsequently cited by Lee~\cite{L90-1}. But the proof in
\cite{SF79} seems to have a gap: the derivation of the differential
equation relies on the smoothness of the function $\varphi (y)$, but
to prove this the authors quote the version of Weyl's Lemma given in
\cite{M69} to conclude that a weak solution must be $C^{\infty}$. We
do not understand this argument, as Weyl's Lemma, in the form stated
in \cite{M69}, applies only to \emph{linear parabolic} differential
operators, while the differential operators in \cite{SF79}, section 2,
and in our Proposition~\ref{prop.varphi.ode} are
\emph{nonlinear}. Therefore, we will sketch another approach to the
proof of Proposition~\ref{prop.varphi.ode} that uses an interesting
class of \emph{product martingales}. (Similar martingales for
supercritical branching Brownian motion were used in \cite{LS88} and
\cite{N88}). Let $h :[0,A]\rightarrow \zz{C}$ be a function
bounded in absolute value by $1$, and denote by $X_{1} (t),\dotsc
,X_{Z (t)} (t)$ the locations of the particles alive at time $t$
(including those frozen at one of the endpoints $0,A$) in a branching
Brownian motion with freezing at $0$ and $A$; define
\begin{equation}\label{eq:product-mg}
	Y (t) =Y_{h} (t) =\prod_{i=1}^{Z (t)} h (X_{i} (t)).
\end{equation}

\begin{proposition}\label{prop.product-mg}
If $h (y)$ satisfies the differential equation $h''=h-\Psi (h)$ in the
interval $(0,A)$ then $Y (t)$ is a bounded martingale, relative to the
standard filtration for the branching Brownian motion, under
$\Prob^{y}$, for any $y\in [0,A]$.
\end{proposition}

\begin{proof}
[Proof of Proposition~\ref{prop.varphi.ode} (Sketch)]
Given Proposition~\ref{prop.product-mg}, we proceed as
follows. Fix $f,g$, and let $h:[0,A]\rightarrow \zz{C}$ be the unique
solution to the boundary value problem 
\begin{gather*}
	\frac{1}{2} h''=h- \Psi (h) ;\\
	h (0)=f (1),\\
	h (A)=g (1).
\end{gather*}
(When $A=\infty$, the boundary condition should be replaced by $h
(A)=g (0)=1$.) The existence and uniqueness of solutions follows by
standard arguments in the theory of ordinary differential equations;
we omit the details.\footnote{At any rate the  argument is routine in
the case where $f (1)$ and $g (1)$ take values in the unit interval
$[0,1]$; in this case existence follows by a routine phase-portrait
analysis for the associated first-order system, using the
nonnegativity of the forcing term $\Psi (h)$. When $f (1)$ and $g (1)$
are complex-valued, however, other methods must be used. See the proof
of Lemma~\ref{lemma.analytic-continuation} in
section~\ref{sec:number-killed} below for a proof in the case needed
for the theorems on the distribution of the number of killed
particles.} By Proposition~\ref{prop.product-mg}, the 
process $Y (t)$ defined by \eqref{eq:product-mg} is a bounded
martingale, and so for any $t<\infty$,
\[
	\E^{y} Y (t)=Y (0)= h (y).
\]
But for all sufficiently large $t$, all particles will be frozen at
either $0$ or $A$, so eventually $Y (t)$ coincides with $f (N_{0})g
(N_{A})$. (For this the product structure of the martingale is
essential.) Therefore, by the bounded convergence theorem,
\[
	h (y)=\E^{y}f (N_{0})g (N_{A}).
\]
\end{proof}

\begin{proof}
[Proof of Proposition~\ref{prop.product-mg} (Sketch)]
By the Markov property, it suffices to show that for any initial
configuration of particles $\mathbf{y}= (y_{1},y_{2},\dotsc ,y_{m})$
the expectation $\E^{\mathbf{y}}Y (t)$ is constant in time. Since each
of the $m$ particles engenders its own independent branching Brownian
motion, the expectation $\E^{\mathbf{y}}Y (t)$ factors  as 
\[
	\E^{\mathbf{y}}Y (t)=\prod_{i=1}^{m}\E^{{y}_{i}}Y (t);
\]
consequently, it suffices to prove that for any $y\in [0,A]$ the
expectation $\E^{y}Y (t)$ is constant in time, and for this it is
enough to show that
\[
	\frac{d}{dt}\E^{y}Y (t)=0.
\] 
But for this another conditioning shows that it is enough to prove
that the derivative is zero at $t=0$.
This can be accomplished by a routine argument, by partitioning the
expectation into the expectations on the events that the initial
particle reproduces or not by time $t$ and using the fact that $h$ is
bounded and $C^{2}$ and satisfies the differential equation
$h''/2=h-\Psi (h)$.
\end{proof}

\section{Weierstrass' $\mathcal{P}$ Functions}\label{sec:weierstrass}

In the special case of double-or-nothing branching (the Moranian
case), the probability generating function of the offspring
distribution is the quadratic function $\Psi (s)=\frac{1}{2}
(1+s^{2})$. In this case the differential equation
\eqref{eqn.varphi.ode} reduces, as we will show, to the differential
equation of the \emph{Weierstrass $\mathcal{P}-$function}. For a given
\emph{period lattice} 
\begin{equation*} 
		  \L=\left\{ m\o+n\wt\o,m,n\in\N \right\}, 
\end{equation*}
where $\omega$ and $\tilde{\omega}$
are nonzero complex numbers whose ratio is not real, Weierstrass'
$\mathcal{P}$ function with period lattice $\L $ is the
meromorphic function on $\zz{C}$ defined by
\begin{equation}\label{eqn.weierstrass-P}
        \P_\L(z)=\frac{1}{z^2}+\sum_{l\in\L, l\neq 0}\left(
	\frac1{(z-l)^2}-\frac1{l^2} \right). 
\end{equation}
See \cite{K84} or \cite{MM99} for expositions of the basic theory.
Clearly, \eqref{eqn.weierstrass-P} defines a doubly-periodic function
of $z$ whose periods are the elements of the lattice $\mathcal{L}$.
It is also evident from \eqref{eqn.weierstrass-P} that
$\mathcal{P}-$functions with proportional period lattices are related
by a scaling law: in particular, for any $\beta \not =0$ and any
lattice $\mathcal{L}$,
\begin{equation}
    \label{eqn.wstrs.scaling}
	\mathcal{P}_{\beta \mathcal{L}} (\beta z)=
	\frac{1}{\beta^{2}}\mathcal{P}_{\mathcal{L}} (z) \quad
	\text{for all} \;\; z\in \zz{C}.
\end{equation}

It is known (cf. \cite{K84} or \cite{MM99}) that the restrictions of
$\mathcal{P}_{\L}$ and its derivative $\mathcal{P}_{\L}'$ to a
fundamental parallelogram are branched covers of the Riemann sphere
$\hat{\zz{C}}$ of degrees $2$ and $3$, respectively, and so for all
but three exceptional values $w\in \zz{C}$ the equation
$\mathcal{P}_{\L} (z)=w$ has two solutions $z_{1},z_{2}$ in each
fundamental parallelogram, and $\mathcal{P}_{\L}'
(z_{1})=-\mathcal{P}_{\L}' (z_{2})$. Furthermore, the function
$\mathcal{P}_{\L} (z)$ satisfies the differential
equation 
\begin{equation} \label{eqn.P.differential}
		 \P_\L'(z)^2=4\P_\L(z)^3-g_2(\L)\P_\L(z)-g_3(\L), 
\end{equation} 
where the constants $g_2(\L)$ and $g_3(\L)$ are given by the
Eisenstein series
\begin{eqnarray*}
    g_2(\L)&=& 60\sum_{l\in \L, l\neq 0}\frac{1}{l^4}\\
    g_3(\L)&=& 140\sum_{l\in \L, l\neq 0}\frac{1}{l^6}.
\end{eqnarray*}
For any two complex numbers $A,B$ such that $A^3-27B^2\neq 0$, there
exists (cf. Proposition III.13 in \cite{K84}) a lattice {$\L$} such that 
\begin{align}\label{eqn.eisenstein-constants}
      g_2(\L)&= A \quad \text{and}\\
 \notag       g_3(\L)&= B.
\end{align}

\begin{proposition}
    \label{prop.u.P} Let $A$ and $B$ be two constants such that
$A^3-27B^2\neq 0$, and Let $u(z)$ be a $C^{1}$
function on an open interval $J\subset \R$ with derivative $u' (x)\not
=0$ for all $x\in J$ that satisfies the
differential equation
    \begin{equation}
        \label{eqn.u.differential}
        u'(z)^2=4u(z)^3-Au(z)-B.
    \end{equation}
    Then for some lattice $\L$ and some  $\alpha\in \zz{C}$,
    \begin{equation}
    \label{eqn.u.P}
        u(x)=\P_\L(x+\alpha) \quad \text{for all} \; x\in J.
    \end{equation}
\end{proposition}

\begin{proof}
Without loss of generality, assume that $0\in J$ and that $u' (0)\not
=0$. The differential equation \eqref{eqn.u.differential} implies that
in some neighborhood of $x=0$, for one of the two branches of the
square root function,
\begin{equation}\label{eqn.ode.sqrt}
	u' (x)=\sqrt{4u(x)^3-Au(x)-B}
\end{equation}
Since $u' (0)\not =0$, the right side of this equation is a Lipshitz
continuous function of $u (x)$ for $x$ near $0$, and so the
Picard-Lindel\"{o}f theorem guarantees that the equation
\eqref{eqn.ode.sqrt} has a unique solution with initial value $u
(0)$.

Let $\L$ be a lattice such that equations
\eqref{eqn.eisenstein-constants} hold.  Because the Weierstrass
$\mathcal{P}-$function is a double covering of $\zz{C}$, there exist
two arguments $\alpha, \alpha ' \in \zz{C}$ such that
$\mathcal{P}_{\L} (\alpha)=\mathcal{P}_{\L} (\alpha ')=u (0)$, and for
one of these (say $\alpha$) it must be the case that
$\mathcal{P}_{\L}' (\alpha)=u' (0)$. Since the functions
$\mathcal{P}_{\L} (x+\alpha)$ and $\mathcal{P}_{\L} (x+\alpha')$ both
satisfy the differential equation \eqref{eqn.u.differential}, one of
them (say $\mathcal{P}_{\L} (x+\alpha)$) must also satisfy
\eqref{eqn.ode.sqrt}. By the Picard-Lindel\"{o}f theorem, the equation 
\eqref{eqn.u.P} must hold in $J$.

 \end{proof}

The connection between the differential equation \eqref{eqn.varphi.ode}
and the Weierstrass $\mathcal{P}-$function is easily explained. If $h
(z)=z^{2}$, then the forcing term in \eqref{eqn.varphi.ode} is
quadratic, and so after a rescaling \eqref{eqn.varphi.ode} can be
written in the form
\begin{equation}\label{eqn.weierstrass.ode-2}
	u ''(y) =6u (y)^{2}.
\end{equation}
Multiplying both sides by $u' (y)$ and integrating yields
\begin{equation}\label{eqn.weierstrass.ode-2.integrated}
	(u' (y))^{2}=4 u (y)^{3}+C,
\end{equation}
where $C$ is a constant of integration. This is the characteristic
equation for a $\mathcal{P}-$function whose period lattice satisfies
$g_{2} (\mathcal{L})=0$.  

\begin{proposition}\label{proposition:antiequianharmonic}
The Weierstrass function $u=\mathcal{P}_{\mathcal{L}}$ satisfies the
differential equation \eqref{eqn.weierstrass.ode-2.integrated} for
some $C\in \zz{C}\setminus \{0\}$ if and only if the period lattice is
of the form 
\begin{equation}\label{eqn.antiequianharmonic}
    \mathcal{L}=\{m\omega +n\omega e^{\pi i/3}\} 
\end{equation}
for some $\omega \not =0$; furthermore,  $C>0$ in
\eqref{eqn.weierstrass.ode-2.integrated} if and only if 
the lattice  has the form \eqref{eqn.antiequianharmonic} with
\begin{equation}\label{eqn.antiequianharmonic-omega}
	\omega =|\omega |e^{\pi i/6}.
\end{equation}
In this case, $u$ has real poles at integer multiples of
$\sqrt{3}|\omega |$, and takes only real values on $\zz{R}$;
furthermore, its only zeros in the fundamental parallelogram are at
$\sqrt{3}|\omega |/3$ and $2\sqrt{3}|\omega |/3$, and $u$ is strictly
increasing on $(2\sqrt{3}|\omega |/3,\sqrt{3}|\omega |)$.
\end{proposition}

\begin{proof}
The Eisenstein series  for the lattice
$\mathcal{L}=\{m\omega  +n \tilde{\omega}\}$ can be written as 
\begin{equation}\label{eqn.eisenstein.modular}
\begin{aligned}
	g_{2} (\mathcal{L})&=60 \sum^{*} _{m,n} (m\omega
	+n\tilde{\omega})^{-4} = 60\omega^{-4}G_{4} (\xi) \quad
	\text{and} \\
        g_{3} (\mathcal{L})&=140 \sum^{*}_{m,n} (m\omega
	+n\tilde{\omega})^{-6} = 140\omega^{-6}G_{6} (\xi)
\end{aligned}
\end{equation}
where $\xi =\tilde{\omega}/\omega$ is the ratio of two fundamental
periods and the sum is over all pairs of integers except $(0,0)$. By
convention, the periods are ordered so that $\Im \xi >0$; with this
convention, $G_{4}$ and $G_{6}$ are modular forms of weights $4$ and
$6$ (cf. \cite{K84}, section III.2). By the residue theorem for
modular forms (cf. \cite{K84}, Proposition III.2.8), any nonzero
modular form of weight $4$ has precisely two zeros in the closure of
the standard fundamental polygon of the modular group, at the points
$\xi_{-} =e^{\pi i/3}$ and $\xi_{+}=e^{2\pi i/3}$. Therefore, any
Weierstrass function $u=\mathcal{P}_{\mathcal{L}}$ that satisfies the
differential equation \eqref{eqn.weierstrass.ode-2.integrated} must
have period lattice of the form \eqref{eqn.antiequianharmonic} (as the
choices $\xi_{-}$ and $\xi_{+}$ lead to the same lattice).

The lattice \eqref{eqn.antiequianharmonic} is invariant under rotation
by $\pi /3$ (that is, $\mathcal{L}=e^{\pi i/3}\mathcal{L}$), and so by
averaging over the six rotations $e^{k\pi i/3}$ one finds that 
\begin{equation}\label{eqn.G6}
	G_{6} (e^{\pi i/3})=G_{6} (e^{2\pi i/3})=\sum_{m,n}^{*}
	\frac{1}{m^{6}+n^{6}} >0. 
\end{equation}
Consequently, if $g_{3} (\mathcal{L})=-C<0$ then $\mathcal{L}$ must be
of the form \eqref{eqn.antiequianharmonic} for some $\omega$ such that
$\omega^{6}<0$, that is, $\omega$ is a positive multiple of a
primitive $12$th root of unity. Thus, in the case $g_{3}
(\mathcal{L})=-C<0$ the lattice $\mathcal{L}$ must have the form
\eqref{eqn.antiequianharmonic} with $\omega = |\omega |e^{\pi i/6}$.

Assume now that $\mathcal{L}$ is of the form
\eqref{eqn.antiequianharmonic} for some $\omega$ satisfying
\eqref{eqn.antiequianharmonic-omega}. Then by the  
\emph{addition law} for the elliptic curve $y^{2}=4x^{3}-g_{3}$
(cf. \cite{K84}, section I.7; see especially Problem 8),
\[
	\mathcal{P}_{\mathcal{L}} (\sqrt{3}|\omega
	|/3)=\mathcal{P}_{\mathcal{L}} (2\sqrt{3}|\omega |/3)=0 . 
\]
Since $\mathcal{P}_{\mathcal{L}}$ has degree $2$, it has only two
zeros in a fundamental parallelogram, and by equation
\eqref{eqn.weierstrass.ode-2.integrated} the derivatives 
$\mathcal{P}_{\mathcal{L}}' (\sqrt{3}|\omega |/3)=-\mathcal{P}_{\mathcal{L}}'
(2\sqrt{3}|\omega |/3)$ must be the two square roots of $C$. It is easily seen
that the unique solution of \eqref{eqn.weierstrass.ode-2.integrated}
with initial conditions $u (y_{0})=0$ and ${u'(y_{0})}>0$  must be
strictly increasing, with increasing derivative, on any interval
$(y_{0},y_{1})$ on which the solution $u$ is well-defined and
finite. This implies that $\mathcal{P}_{\mathcal{L}}' (\sqrt{3}|\omega |/3)$ is
\emph{negative}, and hence $\mathcal{P}_{\mathcal{L}}' (2\sqrt{3}|\omega |/3)$
is \emph{positive}. It then follows that $u$ is strictly increasing in
$(2\sqrt{3}|\omega |/3,\sqrt{3}|\omega |)$. 

\end{proof}

\begin{remark}\label{remark:equianharmonic}
The case where $C=-g_{3} (\mathcal{L})= -1$ in equation
\eqref{eqn.weierstrass.ode-2.integrated} is known as the
\emph{equianharmonic} case; cf. \cite{AS72} for further information.
In the equianharmonic case the period lattice is
of the form \eqref{eqn.antiequianharmonic}, but with $\omega >0$,
i.e., the lattice is of the same form as in the case where $C=1$ but
rotated by $-\pi /6$.   Call the case where $C=-g_{3}
(\mathcal{L})=+1$ the \emph{anti-anharmonic} case; then by the
scaling law, the $\mathcal{P}-$functions for the equianharmonic and
the anti-anharmonic cases are related by
\[
	\mathcal{P}_{AAH} (e^{\pi i/6}z) = e^{-\pi
	i/3}\mathcal{P}_{EAH} (z)  \quad \text{for all} \;\; z\in \zz{C}.
\]
Thus, mapping properties and special values of the
$\mathcal{P}$-function in the anti-anharmonic case can be read off
from those for the equianharmonic case, which have been extensively
tabulated.

As far as we know, the occurrence of the $\mathcal{P}-$function in
critical branching processes was first observed by the first author in
\cite{L09}, sec.~1.8. However, \cite{L09} mistakenly asserts that the
differential equation \eqref{eqn.weierstrass.ode-2.integrated} with
$C>0$ falls into the equianharmonic case, and consequently the
formulas in \cite{L09}, sec.~1.8 are off by factors of $e^{\pi i/6}$
and $e^{\pi i/3}$.
\end{remark}

\section{Distribution of $M$: Moranian Case}\label{sec:M-Moranian}

In this section we consider the  Moranian case \cite{SF79}, where the
number of offspring is either zero or two, each with  probability 
$\frac12$. In this case the probability generating function is
$\Psi(t)=\frac12+\frac12 t^2$. For $0\leq y<x$ define 
\begin{equation}\label{eqn.u.def}
	u_{x} (y)=\Prob^{y}\{M\geq x \}
\end{equation}
to be the probability that the maximum position $M$ attained by a
particle of the branching Brownian motion initiated by a particle at
$y$, with freezing of particles at $0$, will exceed $x$. The function
$\varphi_{x} (y)=1-u_{x} (y)$ is of the form covered by
Proposition~\ref{prop.varphi.ode}, so it satisfies the differential
equation \eqref{eqn.varphi.ode} with $\Psi (z)= (1+z^{2})/2$, and
consequently $u_{x}$ satisfies 
\begin{equation}
    u_x''(y)=2\left[ \frac12+\frac12 (1-u_x(y))^2-(1-u_x(y)) \right]=u_x(y)^2.
    \label{eqn.moranian}
\end{equation}

\begin{theorem}
    \label{thm.wstrs} For branching Brownian motion with Moranian
offspring distribution and killing at $0$, the tail distribution
function $u_{x} (y)=\Prob^{y} (M\geq x)$  is given by
\begin{equation} 
		  u_x(y)=6\P_{\L_x}(y+2\omega_{x}/3),
\label{eqn.wstrs} 
\end{equation} 
where $\P_{\L_x}(z)$ is the
Weierstrass $\P$ function with period lattice 
\begin{equation}
    \L_x=\left\{ m\frac{\omega_{x}}{\sqrt{3}}e^{\pi i/6}+n\frac{\omega_{x}}{\sqrt{3}}e^{\pi i/2}: \, m,n\in
\zz{Z}\right\} 
\label{def.lattice} 
\end{equation} 
for some
$\omega_{x}>0$. The positive period $\omega_{x}$ is uniquely
determined by the boundary condition
\begin{equation}\label{eqn.period.det}
 	6\mathcal{P}_{\mathcal{L}_{x}} (x+2\omega_{x}/3)=1.
\end{equation}
\end{theorem}

\begin{remark}
\label{remark.numerics}
The value of $\omega_{x}$ can be
computed numerically, by exploiting the fact that  the inverse of the
Weierstrass $\mathcal{P}-$function  $w=\mathcal{P}_{\mathcal{L}} (z)$
is given by the elliptic integral 
\begin{equation}\label{eqn.ellipticIntegral}
    z=\int_w^{\infty}\frac{dt}{\sqrt{4t^3-g_2(\L)t-g_3(\L)}}.
\end{equation}
For lattices of the form \eqref{def.lattice}, we have $g_{2}
(\mathcal{L})=0$. Consequently, the
boundary condition \eqref{eqn.period.det} implies that
\begin{equation}\label{eqn.numerics.omega}
    x=\int_{0}^{ \frac{1}{6} } \frac{dt}{\sqrt{4t^3-g_3(\L)}}. 
\end{equation}
This equation \eqref{eqn.numerics.omega} determines  $g_{3}
(\mathcal{L})$, and hence, using the identities
\eqref{eqn.eisenstein.modular}, the value of  $\omega_{x}$. For
$x=1$, the values are 
\begin{equation*}
 g_3(\L_1)=-0.023786\cdots\quad \text{and} \quad 
 \omega_1=9.88285\cdots .
\end{equation*}

The large $x$ dependence of $\omega_{x}$ on $x$  will be further clarified
below, in Corollary~\ref{cor.periodAsymptotics}.
\end{remark}

The proof of Theorem~\ref{thm.wstrs} will rely on the uniqueness
theorem for solutions of the differential equation for the
$\P-$function (Proposition~\ref{prop.u.P}). For this, it will be
necessary to know that  $u_{x}' (y)\not =0$ for any $y\in [0,x]$.

\begin{lemma}\label{lemma.increasing}
The function $y\mapsto u_{x} (y)$ has strictly positive derivative
$u_{x}' (y)$ on the interval $[0,x]$.
\end{lemma}

\begin{proof}
Since $u_{x} (y)$ is a cumulative distribution function, it is
non-decreasing, and so its derivative must be nonnegative. Moreover,
the differential equation $u_{x}'' (y)=u_{x} (y)^{2}$ implies that the
derivative is increasing at every $y$ where $u_{x} (y)>0$. It is
easily seen that a branching Brownian motion started at any  $y>0$ has
positive probability of putting a particle at $x$, so $u_{x} (y)>0$
for all $y\in (0,x]$. 

It remains to show that $u_{x}' (0)>0$.
The differential equation $u''=u^{2}$ can be rewritten as the
autonomous system
\begin{align*}
	u'&=v, \\
	v'&=u^{2}.
\end{align*}
The vector field in this system is clearly Lipshitz continuous, so
solutions to the initial value problem are unique. Since $u\equiv v
\equiv 0$ is the unique solution with initial conditions $u (0)=v
(0)=0$, it follows that any non-constant solution of $u''=u^{2}$
satisfying $u (0)=0$ cannot have derivative $u' (0)=0$. 
\end{proof}

\begin{proof}
[Proof of Theorem~\ref{thm.wstrs}]
    Set $\wt u_x(y)=\frac{1}{6}u_x(y)$; then \eqref{eqn.moranian} becomes
    \begin{equation*}
        \wt u_x(y)''=6\wt u_x(y)^2,
    \end{equation*}
the differential equation encountered earlier in
\eqref{eqn.weierstrass.ode-2}. The integrated form is
\eqref{eqn.weierstrass.ode-2.integrated}. 
By Lemma~\ref{lemma.increasing}, the derivative $u_{x}' (y)$ is strictly
positive on $y\in [0,x]$, and so the same is obviously true of $\wt
u_{x} (y)$. Hence, by Proposition~\ref{prop.u.P}, 
$\wt u_x$ must coincide with a translate of a $\mathcal{P}-$function,
and so for each $x>0$ there exists a
unique period lattice and a unique $\alpha_{x}\in \zz{C}$ such that 
\[
	\wt u_{x} (y)=\mathcal{P}_{\mathcal{L}} (y+\alpha_{x}) \quad
	\text{for all}\;\; y\in [0,x].
\]
Since there is no linear term in the equation
\eqref{eqn.weierstrass.ode-2}, the period lattice $\mathcal{L}$ must
be of the form \eqref{def.lattice}.

    Finally, the boundary condition $u_{x} (0)=0$ and
\eqref{eqn.wstrs} imply that $\alpha_x$ is a zero for
$\P_\L(z)$. Since $u_x(y)$ is increasing in $y$ and $\P_\L(z)$ is
doubly periodic, the constant $\alpha_x$ must be the larger zero of $\P_\L(z)$
in $(0, \omega_{x})$, and in particular, by
Proposition~\ref{proposition:antiequianharmonic},
\begin{equation}\label{eqn.alpha-x}
	\alpha_{x}=2\omega_{x}/3.
\end{equation}
\end{proof}

All of the $\mathcal{P}-$functions that occur in
Theorem~\ref{thm.1} are scaled versions of
$\mathcal{P}_{\mathcal{L}_{1}}$. 
By \eqref{eqn.wstrs.scaling}, 
if the positive periods $\omega_{1}$ and $\omega_{x}$ of
    the lattices $\mathcal{L}_{1}$ and $\mathcal{L}_{x}$, respectively,
are related by 
\begin{equation}\label{eqn.def.lambda-x}
	\lambda_{x}:= \frac{\omega_{1}}{\omega_{x}},
\end{equation}
then 
\begin{equation}\label{eqn.wstrs.scalingLaw}
\mathcal{P}_{\mathcal{L}_{x}}(z)
 =\lambda_{x}^{2}\mathcal{P}_{\mathcal{L}_{1}} (\lambda_{x}z).
\end{equation}

It is obvious that $\lim_{x \rightarrow \infty}\omega_{x}= \infty$,
because the function $u_{x} (y)$ is an increasing function on $(0,x)$,
and hence cannot have a positive period smaller than $x$. Although the
dependence of $\omega_{x}$ on $x$ is not linear, it is asymptotically
linear, as the next corollary shows.

\begin{corollary}\label{cor.periodAsymptotics}
\begin{equation}\label{eqn.asy.lin}
\lim_{x \rightarrow \infty}\frac{\omega_{x}}{x}
{=\lim_{x \rightarrow \infty}{\frac{\omega_1}{\lambda_{x}x}}}
=3.
\end{equation}
\end{corollary}

\begin{proof}
The scaling law \eqref{eqn.wstrs.scaling} and the boundary
conditions for the functions $u_{x}$ and $u_{1}$ imply that
\begin{equation}
    \label{eqn.wstrs.max}
    \P_{\L_1}(2\omega_{1}/3+\lambda_x x)=\frac1{6\lambda_x^2}.
\end{equation}
Since $(6\lambda_{x}^{2})^{-1}\rightarrow \infty$ as $x \rightarrow \infty$, it
follows that $2\omega_{1}/3+\lambda_x x$ converges to $\omega_{1}$, as this
is the smallest positive pole of $\mathcal{P}_{\mathcal{L}_{1}}$.
The result now follows from the equation  \eqref{eqn.def.lambda-x}.
\end{proof}

\begin{remark}\label{remark:expansion}
By exploiting the fact that $\mathcal{P}_{\mathcal{L}_{1}} (\omega_{1}-z)\sim
1/z^{2}$ as $z \rightarrow 0$, one can obtain from the equation
\eqref{eqn.wstrs.max}  the  sharper approximation
\begin{equation}\label{eqn.sharpPeriodAsymptotics}
	\omega_{x}=3x+3/\sqrt{6} +o (1) \quad
	\text{as} \;\; x \rightarrow \infty.
\end{equation}
\end{remark}

The scaling laws \eqref{eqn.wstrs.scalingLaw} and the period
asymptotics \eqref{eqn.asy.lin} now combine to provide the large$-x$
asymptotic behavior of the hitting probability function $u_{x} (y)$.

\begin{theorem}
For branching Brownian motion with Moranian
offspring distribution and killing at $0$, the tail distribution
function $u_{x} (y)=\Prob^{y} (M\geq x)$ of the maximum attained
position $M$ satisfies 
    \begin{equation}
        \lim_{x\to\infty}x^3u_x(y)=C_1\cdot y \quad \text{for each}\;\;y>0,
        \label{eqn.result1}
    \end{equation}
    where $C_1=6c_1^3\P_{\L_1}'(2\omega_{1}/3)=33.0822\cdots$ and $c_1=\omega_1 /3=3.29428\cdots$ are constants that do not depend on $x$ or $y$. Furthermore, for each  fixed $0<s<1$,
    \begin{equation}
        \lim_{x\to\infty}x^2u_x(sx)=C_2 (s)
        \label{eqn.result2}
    \end{equation}
    where $C_2 (s) =6c_1^2\P_{\L_1}(2\omega_1/3+sc_1)$.
    \label{thm.1}
\end{theorem}

\begin{proof}[Proof of Theorem \ref{thm.1}]
By equations \eqref{eqn.wstrs} and \eqref{eqn.moranian}, the function
$\mathcal{P}_{\mathcal{L}_{1}}$ satisfies the second-order
differential equation 
\begin{equation}
    \P_{\L_1}''(z)=6\P_{\L_1}(z)^2.
    \label{eqn.wstrs.scaled.ode}
\end{equation}
Using the abbreviation $\alpha_{x}=2\omega_{x}/3$ and the fact that
$\mathcal{P}_{\mathcal{L}_{1}} (\alpha_{1})=0$
(cf. Proposition~\ref{proposition:antiequianharmonic}), it follows by 
taking successive derivatives that
\begin{eqnarray*}
    \P''_{\L_1}(\alpha_1)&=& 0, \\
    \P'''_{\L_1}(\alpha_1)&=& 0, \quad \text{and}\\
    \P^{(4)}_{\L_1}(\alpha_1)&=&12\left(\P'_{\L_1}(\alpha_1)\right)^2 .
\end{eqnarray*}
Lemma~\ref{lemma.increasing} implies that $\P'_{\L_1}(\alpha_1)>0$,
since this is proportional to $u_{1}' (0)$.
Consequently, since $\lambda_{x} \rightarrow 0$ as $x \rightarrow
\infty$,  Taylor expansion around the point $\alpha_{1}$ yields
\begin{eqnarray*}
    u_x(y)=6\P_{\L_x}(\alpha_x+y)&=& 6\lambda_x^2\P_{\L_1}(\alpha_1+\lambda_xy)\\
    &=& 6\lambda_x^2\left(\P_{\L_1}'(\alpha_1)\lambda_xy+\P_{\L_1}^{(4)}(\alpha_1)\lambda_x^4y^4+\cdots\right)\\
    &=& 6\P_{\L_1}'(\alpha_1)\lambda_x^3y+O(\lambda_x^{6}),
\end{eqnarray*}
Therefore, by Corollary~\ref{cor.periodAsymptotics},
\begin{equation*}
    \lim_{x\to\infty} x^3u_x(y)=6c_1^3\P_{\L_1}'(\alpha_1)y=C_1y.
\end{equation*}
and \eqref{eqn.result1} follows.

To prove \eqref{eqn.result2}, notice that for $y=sx$, we have
\begin{equation*}
    x^2u_x(y)=6x^2\P_{\L_x}(\alpha_x+y)= 6(\lambda_xx)^2\P_{\L_1}(\alpha_1+s(\lambda_xx))
\end{equation*}
 By the continuity of $\P_{\L_1}(z)$ on the
interval $(\alpha_1,\o_1)$, it follows  that
\begin{equation*}
    \lim_{x\to \infty}x^2u_x(y)=6c_1^2\P_{\L_1}(\alpha_1+sc_1)=C_2(s).
\end{equation*}
\end{proof}

\section{Distribution of $M$:  General Case}\label{sec:M-general}

In this section we will show that the asymptotic formulas
\eqref{eqn.result1} and \eqref{eqn.result2} extend to branching
Brownian motions with arbitrary mean $1$ offspring
distributions with finite third moments. The main result is as follows.

\begin{theorem}
Assume that the offspring distribution has mean $1$, positive variance
$\sigma^2$, and finite third moment.  Then for each $y>0$, the
probability  $u_x(y)=\Prob^y \{M\geq x \}$ that a particle of the
branching Brownian motion with initial particle at location $y$
reaches location $x$ satisfies
    \begin{equation}
        \lim_{x\to\infty}x^3u_x(y)= C_3\cdot y,
        \label{eqn.result3}
    \end{equation}
    where $C_3={C_1}/{\sigma^2}$.
    In addition, for each fixed $0<s<1$,
    \begin{equation}
        \lim_{x\to \infty}x^2u_x(sx)= C_4,
        \label{eqn.result4}
    \end{equation}
    where $C_4={C_2}/{\sigma^2}$. Here $C_{1},C_{2}$ are the constants
    in Theorem~\ref{thm.1}.
    \label{thm.max.general}
\end{theorem}

This theorem will be deduced from Theorem~\ref{thm.1} by comparison
arguments for differential equations. The strategy is similar to that
used by Lee~\cite{L90-1}: the  key   is that
for small values of $u_{x} (y)$, the forcing term $h (u_{x} (y))$ in
the differential equation 
\begin{equation}\label{eqn.u.ode}
	u_{x}'' = h (u_{x})
\end{equation}
(which follows from Proposition~\ref{prop.varphi.ode} as in the
Moranian case) is well-approximated by the quadratic function
$\sigma^{2}u_{x} (y)^{2}$. To see this, let $\Psi(z)$ be the
probability generating function of the offspring distribution. If the
offspring distribution has finite third moment, then by Taylor
expansion
\begin{equation*}
    \Psi(1-z)=\Psi(1)-\Psi'(1)z+\frac12\Psi''(1)z^2+O(z^3) \quad \text{as}\; z\to 0.
\end{equation*}
By hypothesis, the offspring distribution has mean $1$ and positive
variance $0<\sigma^{2}<\infty$, so $\Psi'(1)=1$ and $\Psi''(1)=\sigma^2.$
Consequently, as $z \rightarrow 0$,
\begin{equation}
    \label{eqn.h.general}
    h(z)=2\left[ \Psi(1-z)-(1-z) \right]=\sigma^2z^2+O(z^3).
\end{equation}

Our arguments will use the following  comparison principle for solutions
to differential equations and inequalities. This is a minor
modification of the Comparison Lemma in \cite{L90-1}; because the
result is standard and its proof involves only elementary
calculus, we shall omit it.

\begin{lemma}[Comparison Principle]
    \label{lma.comparison.2}
    Let $v_{1},v_{2}$ be positive functions on an interval $[y_{1},y_{2}]$
     such that  for some  constant $a>0$,
\begin{align}\label{eqn.super,sub}
	v_1''(y)-av_1(y)^2&\leq 0 \quad \text{and}\\
\notag  v_2''(y)-av_2(y)^2&\geq 0
\end{align}
    for all $y_{1}<y<y_{2}$. If 
\begin{align}\label{ineqn.boundary}
	v_1(y_1)&\geq     v_2(y_1) \quad \text{and}\\
	v_1(y_2)&\geq v_2(y_2),
\end{align}
then
    \begin{equation}
        v_1(y)\geq v_2(y) \quad \text{for all} \;\; y_1\leq y\leq y_2.
        \label{eqn.comparison}
    \end{equation}
\end{lemma}

Next, we record  several monotonicity
properties of the functions $u_{x} (y) =\Prob ^{y}\{M\geq x \}$.

\begin{proposition}
    \label{prop.mono1}
    The function $u_x(y)$ is strictly decreasing in $x$ and strictly 
    increasing in $y$, and $u_x'(0)$ is strictly  decreasing in $x$.
\end{proposition}

\begin{proof}
    The monotonicity of $u_x(y)$ in $y$ and in $x$ follow directly
from the  construction of the branching
Brownian motion outlined in section~\ref{sec:construction}.

To prove that $u_x'(0)$ is strictly decreasing in $x$, recall
 that for each $x>0$ the function
$u_{x}$ satisfies the differential equation $u''=h (u)$, together with
the boundary conditions $u_{x} (0)=0$ and $u_{x} (x)=1$.  Hence, by
the uniqueness theorem for differential equations, if $0<x_{1}<x_{2}$
then $u_{x_{1}}' (0)\not =u_{x_{2}}' (0)$, because otherwise the
functions $u_{x_{i}} (y)$ would be equal for all $y\in [0,x_{1}]$,
which is impossible because $u_{x_{2}} (x_{1})<1$.

Thus, to complete the proof it suffices to show that if
$0<x_{1}<x_{2}$ then $u_{x_1}'(0)< u_{x_2}'(0)$ is impossible. But
since $u_{x_{1}} (0)=u_{x_{2}} (0)=0$,  if 
$u_{x_1}'(0)< u_{x_2}'(0)$ then for all $y$ in some neighborhood
$(0,\varepsilon)$ we would have $u_{x_{1}} (y)<u_{x_{2}} (y)$. This
would contradict the monotonicity of $u_{x} (y)$ in $x$.
\end{proof}

To study the behavior of $u_x(y)$ near $y=x$, we introduce the
function 
\begin{equation}
    w_x(t)=u_x(x-t)
    \label{def.flip}
\end{equation}
Observe that $w_{x} (t)$ is the probability that a branching Brownian
motion with killing at $- (x-t)$ and initial particle at $0$ will
produce a particle that reaches location $t$. The construction in
section~\ref{sec:construction} shows that $w_{x} (t)$ is
\emph{strictly} monotone in both $x$ and $t$, and that 
\begin{equation}\label{eqn.lim.wx}
	\lim_{ x \rightarrow \infty} w_{x} (t)=w_{\infty} (t)
\end{equation}
where $w_{\infty } (t)$ is the probability that a branching Brownian
motion with \emph{no} killing started at location $0$ will produce a
particle that reaches location $t$
(cf. equation~\eqref{eqn.no-kill-lim}). The convergence
\eqref{eqn.lim.wx} holds uniformly for $t$ in any finite interval
$[0,t_{*}]$. The function $w_\infty(t)$ is the same as the function
$p(x)$ in the $d=1$ case studied in \cite{SF79}, who proved that
\begin{equation}\label{eqn.p.asymptotics}
    w_\infty(t)=\frac{6}{\sigma^2 t^2}+O(\frac{1}{t^3}) \quad
    \text{as} \;\; t\to \infty.
\end{equation}

To prove Theorem \ref{thm.max.general}, we must determine the behavior
of $w_{x} (t)$ for \emph{large} $t$, and in particular for $t$ within
distance $O (1)$ of $x$. The basic strategy will be as follows. For
any $\varepsilon >0$ there exists $t_{*}<\infty$ so large that
$w_{\infty} (t_{*})<\varepsilon$. This implies that $w_{x}
(t_{*})<\varepsilon $, or equivalently $u_{x} (x-t_{*})<\varepsilon$,
for all large $x$. Thus, in the interval $[0,x-t_{*}]$ the function
$u_{x}$ will be bounded above by $\varepsilon$, and so $h (u_{x})$
will be well-approximated by the quadratic function
$\sigma^{2}u_{x}^{2}$. The analysis of
sections~\ref{sec:weierstrass}--\ref{sec:M-Moranian} shows that the
differential equation~\eqref{eqn.varphi.ode} with $h (u)=Cu^{2}$
admits an exact solution in terms of a Weierstrass
$\mathcal{P}-$function, so it will follow from the comparison
principle above that in the interval $[0,x-t_{*}]$ the function
$u_{x}$  will be trapped between two such $\mathcal{P}-$functions. By
taking $\varepsilon  \rightarrow 0$, we will obtain sharp asymptotic
approximations to $u_{x}$.

Taylor expansion of $h(z)$ shows that
for all $0<\delta<1$ there exists $\varepsilon=\varepsilon(\delta)>0$
such that if $0\leq u_x(y)\leq \varepsilon$, then   
\begin{equation}
    \sigma^2(1-\delta)u_x^2(y)\leq h\left( u_x(y) \right)\leq \sigma^2(1+\delta)u_x^2(y).
    \label{eqn.h.comparison}
\end{equation}
On the other hand, \eqref{eqn.p.asymptotics} and the monotonicity of
$w_\infty(t)$ imply that for each $\varepsilon\in (0,1)$, there exists
$t_{\varepsilon}$ such that $w_\infty(t_{\varepsilon})=\varepsilon$ and
$w_\infty(t)\leq \varepsilon$ for all $t\geq t_{\varepsilon}$. By 
\eqref{eqn.lim.wx},  for all $0\leq y\leq
x-t_{\varepsilon}$, 
\begin{equation*}
    u_x(y)=w_x(x-y)\leq w_\infty(x-y).
\end{equation*}
 Therefore, \eqref{eqn.h.comparison} applies for all $0\leq y\leq
 x-t_{\varepsilon}$. Define 
\[
	\eta(x,\varepsilon)=u_x(x-t_{\varepsilon});
\]
then for each $\varepsilon >0$ the function 
 $\eta(x,\varepsilon)$ is increasing in $x$, and so
\begin{equation}
    \eta(x,\varepsilon)=w_x(t_{\varepsilon})\uparrow
    w_\infty(t_{\varepsilon})=\varepsilon \quad  \text{ as }x\rightarrow \infty.
    \label{asym.eta}
\end{equation}

\begin{corollary}[Pinching]
    \label{cor.pinching}
    Let $t_{\varepsilon}$ and $\eta(x,\varepsilon)$ be  as above, and define
    \begin{equation}
        \label{def.a+-}
        \left\{
        \begin{aligned}
            a_+&=\sigma^2(1+\delta)>0\\
            a_-&=\sigma^2(1-\delta)>0.
        \end{aligned}
        \right.
    \end{equation}
    If $u_x(y), u_x^+(y)$ and $u_x^-(y)$ satisfy the  boundary value problems
    \begin{equation}
        \label{eqn.u_x}
        \left\{
        \begin{aligned}
            &u_x''(y)= h(u_x(y))\\
            &u_x(0)=0\\
            &u_x(x-t_{\varepsilon})=\eta(x,\varepsilon),
        \end{aligned}
        \right.
        \quad
        \left\{
        \begin{aligned}
            &u_x^{\pm\prime\prime}(y)= a_{\pm}u_x^{\pm 2}(y)\\
            &u_x^{\pm}(0)=0\\
            &u_x^{\pm}(x-t_{\varepsilon})=\eta(x,\varepsilon),
        \end{aligned}
        \right.
    \end{equation}
    then
\begin{equation}\label{eq.ux.pinched}
	 u_x^+(y)\leq    u_x(y)\leq u_x^-(y) \quad \text{for all} \;\; 0\leq y\leq x-t_{\varepsilon}.
\end{equation}
\end{corollary}

\begin{proof}
This is an immediate consequence of the comparison principle
(Lemma~\ref{lma.comparison.2}). 
\end{proof}

The differential equations \eqref{eqn.u_x} for the functions
$u^{\pm}_{x}$ are, except for the  constants $a_{\pm}$, identical to
the differential equation \eqref{eqn.weierstrass.ode-2} for the
$\mathcal{P}-$function. Consequently, they are related by a simple
scaling law.

\begin{lemma}
    \label{lma.u_x.pm.scaled}
    Suppose that  the functions $u_x^{\pm}(y)$ satisfy
    \eqref{eqn.u_x}, and set
    \begin{equation}
        \wh u_x^{\pm}(y)=u_x^{\pm}\left(\frac{y}{\sqrt{a_\pm}}\right).
        \label{def.u_x.pm.scaled}
    \end{equation}
    Then the functions $  \wh u_x^{\pm}$ satisfy the boundary value problems
    \begin{equation}
        \left\{
        \begin{aligned}
            &\wh u_x^{\pm\prime\prime}(y)= \wh u_x^{\pm 2}(y)\\
            &\wh u_x^{\pm}(0)=0\\
            &\wh u_x(\sqrt{a_\pm}(x-t_{\varepsilon}))= \eta(x,\varepsilon).
        \end{aligned}
        \right.
        \label{eqn.u_x.pm.scaled}
    \end{equation}
\end{lemma}

\begin{proof}
    This is an immediate consequence of the definition
    \eqref{def.u_x.pm.scaled} and equation \eqref{eqn.u_x}. 
\end{proof}

\begin{proof}[Proof of Theorem \ref{thm.max.general}]
The differential equation in \eqref{eqn.u_x.pm.scaled} is the same as
in the Moranian case, so
\begin{equation*}
    \wh u_x^\pm(y)=6\P_{\L_x^\pm}\left(\alpha_x^\pm+y\right),
\end{equation*}
where the  period lattice $\L_x^\pm$ has fundamental periods
$\omega_x^\pm$ and $\omega_x^\pm e^{2\pi i/3}$. The centering constant 
$\alpha_x^{\pm}$ is the larger zero of $\P_{\L_x^{\pm}}(z)$ on
$(0,\omega_x^\pm)$. By the scaling laws for the $\mathcal{P}-$functions,
\begin{equation}
    \label{eqn.u_x.pm.P}
    \wh u_x^\pm(y)=6\lambda_x^{\pm 2}\P_{\L_1}\left(\alpha_1+\lambda_x^{\pm}y\right).
\end{equation}
where $ \lambda_x^\pm={\omega_1}/{\omega_x^{\pm}}$.
The boundary conditions at $y=\sqrt{a_\pm}(x-t_{\varepsilon})$ in
\eqref{eqn.u_x.pm.scaled} and \eqref{eqn.u_x.pm.P} imply that 
\begin{equation}
    \label{eqn.u_x.pm.boundary}
    \P_{\L_1}(\alpha_1+\sqrt{a_\pm}\lambda_x^\pm(x-t_{\varepsilon})=\frac{\eta(x,\varepsilon)}{6\lambda_{x}^{\pm 2}}.
\end{equation}
By
\eqref{asym.eta}, $\lim_{x \rightarrow \infty}\eta(x,\varepsilon)=
\varepsilon$, and furthermore $x-t_{\varepsilon}\sim x$,  since
$t_{\varepsilon}=O(1)$ as $x\to \infty$. Hence, $\lambda_x^{\pm}\to
0$,  by \eqref{eqn.def.lambda-x}. Consequently, the right hand side of
\eqref{eqn.u_x.pm.boundary} goes to $\infty$ as $x\to \infty$. It
follows that $\alpha_1+\sqrt{a_\pm}\lambda_x^\pm x\to \o_1$, as $x\to
\infty$. Therefore, 
\begin{equation}
    \lim_{x\to\infty}\lambda_x^\pm x=\frac{\o_1-\alpha_1}{\sqrt{a_\pm}}=\frac{c_1}{\sqrt{a_\pm}}.
    \label{asym.lamda.pm}
\end{equation}

Equations \eqref{def.u_x.pm.scaled} and  \eqref{eqn.u_x.pm.P} imply that
\begin{equation}
    \label{eqn.v.solution}
    u_x^\pm(y)=\wh u_x^\pm(\sqrt{a_\pm}y)=6\lambda_x^{\pm 2}\P_{\L_1}(\alpha_1+\lambda_x^\pm\sqrt{a_\pm}y).
\end{equation}
Consequently, when $0\leq y\leq x-t_{\varepsilon}$ is fixed, Taylor
expansion of $\P_{\L_1}(z)$ around $z=\alpha_1$ as in the
Moranian case yields 
\begin{eqnarray*}
    \lim_{x\to\infty}x^3u_x^\pm(y)&=&\lim_{x\to\infty} 6x^3\lambda_x^{\pm 2}\P_{\L_1}(\alpha_1+\sqrt{a_\pm}\lambda_x^\pm y)\\
    &=& \lim_{x\to\infty}6(\lambda_x^\pm x)^3\sqrt{a_\pm}\P_{L_1}'(\alpha_1)y+O(\lambda_x^{\pm 3})\\
    &\stackrel{\eqref{asym.lamda.pm}}{=}&\frac{6c_1^3\P_{\L_1}'(\alpha_1)y}{a_\pm}= \frac{C_1 y}{a_\pm}.
\end{eqnarray*}
Similarly, if $y=sx$ for some fixed $0<s<1$, then by the continuity of
$\P_{\L_1}(z)$ on $(\alpha_1,\omega_1)$ and \eqref{asym.lamda.pm},  
\begin{eqnarray*}
    \lim_{x\to\infty}x^2u_x^\pm(sx)&=&\lim_{x\to\infty} 6(x\lambda_x^\pm)^2\P_{\L_1}(\alpha_1+s\sqrt{a_\pm}\lambda_x^\pm x)\\
    &\stackrel{\eqref{asym.lamda.pm}}{=}& \frac{6c_1^2\P_{\L_1}(\alpha_1+sc_1)}{a_\pm}=\frac{C_2(s)}{a_\pm}.
\end{eqnarray*}
Here $C_1$ and $C_2(s)$ are  as in Theorem \ref{thm.1}.

Finally,
by Corollary  \ref{cor.pinching},
\begin{equation*}
    \begin{aligned}
        \text{ for each $y$ fixed },\quad &\frac{C_1y}{a_+}\leq \lim_{x\to \infty}x^3u_x(y)\leq \frac{C_1y}{a_-};  \;\; \text{and}\\
        \text{ for each $0<s<1$ fixed },\quad &\frac{C_2(s)}{a_+}\leq \lim_{x\to \infty}x^2u_x(sx)\leq \frac{C_2(s)}{a_-}.
    \end{aligned}
\end{equation*}
Letting $\delta\to 0$, so that $a_\pm\to \sigma^2$, we obtain \eqref{eqn.result3} and
\eqref{eqn.result4}.
\end{proof}

\section{The Number of Killed Particles}\label{sec:number-killed}

In this section we discuss the distribution of the number $N=N_0$ of
particles killed during the course of a critical branching Brownian
motion with killing at $0$ initiated by a single particle at position
$y>0$. Our primary interest is in the tail of the distribution, that
is, in the large-$k$ behavior of the probabilities $\Prob^{y}\{N \geq
k\}$.  

\begin{theorem}\label{theorem.killed-number.weak}
If the offspring distribution has mean $1$,  positive variance
$\sigma^2$, and finite third moment then 
\begin{equation}\label{eqn.killed-number.weak}
	\sum_{k=1}^{m} k\Prob (N\geq k) \sim 2C_{7}y\sqrt{m} \quad
	\text{ where } \;\; C_7=\frac{\sigma}{\sqrt{6\pi}}. 
\end{equation}
\end{theorem}

The proof, which will use a form of Karamata's Tauberian theorem,
will be given in section~\ref{ssec.killed-number.weak}.  

If we knew that the sequence $k\Prob^{y} (N\geq k)$ were monotone then
we could conclude from \eqref{eqn.killed-number.weak} that $\Prob^{y}
(N\geq k)\sim C_{7}y/k^{3/2}$. However, it seems unlikely that
monotonicity of the sequence $k\Prob^{y} (N\geq k)$ holds in
general.  Thus, to obtain sharp asymptotic results about the individual probabilities
$\zz{P}^{y} (N=k)$, we will impose more restrictive hypotheses on the
offspring distribution that will allow us to avoid the use of
Karamata's theorem. In its place, we will use a result of Flajolet and
Odlyzko \cite{FO90} that allows one to extract information about the
asymptotic behavior of the coefficients of a power series from
information about its behavior on the circle $\Gamma$ of
convergence. Our hypotheses are
most conveniently formulated in terms of the functions
\begin{equation}\label{eqn.h.kappa}
	h (s)=2[\Psi (1-s)- (1-s)] \quad \text{and} \quad 
	\kappa (s) =\int_{0}^{s} h (s')\,ds',
\end{equation}
where $\Psi$ is the probability generating function of the offspring
distribution.  Since the power series for a probability generating
function has radius of convergence $1$, the function $h$ extends to an
analytic function $h (z)$ in the disk of radius $1$ centered at $z=1$,
as does its integral $\kappa$. If the offspring distribution has
finite support then the functions $h$ and $\kappa$ are
polynomials, and consequently are well-defined and analytic in the
entire plane $\zz{C}$. Observe that $\kappa$ has a zero of degree $3$
at $s=0$, since $h (s)=s^{2}+O (|s|^{3})$.

\begin{theorem}
\label{thm.killed.asym.general}
Assume  that the offspring distribution has mean $1$, positive
variance $\sigma^{2}$,  and that the function $h (z)$ extends
analytically to a disk of radius $2+\varepsilon$ centered at $0$. If
the indefinite integral $\kappa
(s)$ has no zeros in the punctured  disk $0<|s|\leq 2$,
then  for each $y>0$,  
    \begin{equation}
        \Prob^y(N\geq k)\sim \frac{C_7 y}{k^{3/2}}
	\quad \text{ where } \;\; C_7=\frac{\sigma}{\sqrt{6\pi}}.
        \label{eqn.v.asymp.general}
    \end{equation}
and
\begin{equation}\label{eqn.p.asymp}
	\Prob^y(N= k)\sim \frac{C_{8}y}{k^{5/2}}\quad \text{ where } \;\;
        C_8=\frac{3\sigma}{2\sqrt{6\pi}}.
\end{equation}
\end{theorem}

For the double-or-nothing (Moranian) offspring distribution, the
functions $h (z)$ and $\kappa (z)$ are given by $h (z)=z^{2}$ and
$\kappa (z)=z^{3}/3$, and so the conclusions of
Theorem~\ref{thm.killed.asym.general} hold. (In this case, we will
exhibit an explicit closed-form representation of the distribution, in
Theorem~\ref{thm.killed.exact.moranian} below.)  Consequently, by
Rouche's theorem, the hypotheses of
Theorem~\ref{thm.killed.asym.general} hold for all finitely-supported
offspring distributions in a neighborhood of the double-or-nothing
distribution, in the following sense: for any integer $m\geq 3$ there
exists $\alpha_{m}\in (0,1/2)$ such that the hypotheses of
Theorem~\ref{thm.killed.asym.general} hold for any probability
distribution $\{q_{n} \}_{0\leq n\leq m}$ such that $\min
(q_{0},q_{2})>\alpha_{m}$.

Theorem \ref{thm.killed.asym.general} should be compared with recent
results of Maillard~\cite{M13} and Berestycki \emph{et al.}
\cite{BBHM15}, which give sharp tail probability estimates for the
number of killed particles in the somewhat different context of
supercritical branching Brownian motion with particle drift. Both
\cite{M13} and \cite{BBHM15} also  use the Flajolet-Odlyzko theorem,
and so must also contend with the issue of analytic continuation of
the generating function. In \cite{BBHM15}, the reproduction mechanism
is simple binary fission, and so there is no need to impose additional
conditions. In \cite{M13}, the offspring distribution is arbitrary,
but must have exponentially decaying tails and mean greater than
$1$. In all three cases, it would be of interest to determine optimal
hypotheses on the offspring distribution.

The proof of Theorem \ref{thm.killed.asym.general} will be given in
section~\ref{ssec.killed.general}.  In sections
\ref{ssec:expNumberKilled}, \ref{ssec.gf}, and
\ref{ssec.killed-number.weak}, we shall assume only that the
hypotheses of Theorem~\ref{theorem.killed-number.weak} are in force;
in section~\ref{ssec:killed-moranian} we shall assume that the
offspring distribution is the double-or-nothing distribution; and in
section~\ref{ssec.killed.general} we shall assume that the offspring
distribution satisfies the hypotheses of Theorem \ref{thm.killed.asym.general}.

\subsection{Expected number of killed particles}\label{ssec:expNumberKilled} 

\begin{proposition}
    \label{prop.exp.N}
    \begin{equation}
        \E^y[N]=1.
        \label{eqn.exp.N}
    \end{equation}
\end{proposition}

The proof will use the following \emph{a priori}
bound on the expectation.

\begin{lemma}
    \label{lma.bound.exp.N}
    \begin{equation*}
        \E^y[N]\leq 1
    \end{equation*}
\end{lemma}

\begin{proof}
     For each time $t$, let $N_t$ denote the total number of particles
frozen at $0$ by time $t$, and $Z_t$  the total number of particles in
the enveloping branching Brownian motion with no particle freezing (see the
construction in section~\ref{sec:construction}).  The counting
process $\left\{ N_t \right\}_{t>0}$ is clearly 
increasing in $t$, and
    \begin{equation*}
        \lim_{t\to \infty}N_t=N.
    \end{equation*}
Consequently, by the monotone convergence theorem, it will suffice to
show that $\E^y[N_t]\leq 1$ for any $t>0$.

 Let $\wt Z_t$ be the
total number of particles at time $t$ in 
the branching Brownian motion with freezing (including those particles
frozen at $0$),  and $W_t$  the initial particle's location at
time $t\geq 0$. Define $T$ to be the time of the first reproduction
event (recall that this is a unit exponential random variable
independent of the  branching Brownian motion) and $\tau_{0}$ the
first time that a particle reaches the origin. Since the offspring
distribution has mean $1$,
\[
	E^{y} Z_{T\wedge \tau_{0}}=1.
\]
Now a particle that reaches zero will, in the enveloping branching
Brownian motion, engender a \emph{critical} descendant branching Brownian
motion, and so by the strong Markov property,
\begin{equation*}
    \E^0[Z_{t- (T\wedge \tau_0)}]=1.
\end{equation*}
This implies that $E^y[\wt Z_t]=\E^y[Z_t]=1$. Clearly $N_t\leq \wt
Z_t$, so $\E^y[N_t]\leq \E^y[\wt Z_t]=1$ for any $t>0$. 
\end{proof}

\begin{proof}[Proof of Proposition \ref{prop.exp.N}]
    Let
    \begin{equation}
        Z^*_t=\wt Z_t-N_t
        \label{def.Z*}
    \end{equation}
be the number of particles of the branching Brownian motion alive at
time $t$ that are not frozen. Since
$\E^y[\wt Z_t]= 1$ for any $t>0$, it is enough to show that
$EZ^{*}_{t}\rightarrow 0$ as $t \rightarrow \infty$. 

Clearly, $Z^{*}_{t}\leq Z_{t}$, because the particles counted in
$Z^{*}_{t}$ are contained in the set of particles counted by
$Z_{t}$. In fact, 
\begin{align*}
	Z^{*}_{t}&=\sum_{i=1}^{Z_{t}} \mathbf{1}\{\text{particle $i$
	trajectory}\; \subset (0,\infty)\} \; \Longrightarrow \\
	E^{y}Z^{*}_{t}&= E^{y}\sum_{i=1}^{Z_{t}} \mathbf{1}\{\text{particle $i$
	trajectory}\; \subset (0,\infty)\}.
\end{align*}

To evaluate the last expectation, we use the discrete Brownian snake
representation of branching Brownian motion described in
section~\ref{sec:construction}. Recall that in this construction
particles are represented by vertices of a Galton-Watson tree, and
their locations are obtained by running conditionally independent
Wiener processes along the edges. Thus, 
for any particle $i$ counted in $Z_{t}$, the conditional distribution
of the trajectory $\{W^{i}_{s} \}_{s\leq t}$ of particle $i$ up to
time $t$, given the realization of the skeletal branching process, is that
of Brownian motion started at $y$.  Hence,
\begin{align*}
	E^{y}Z^{*}_{t}&= E^{y}\sum_{i=1}^{Z_{t}} \mathbf{1}\{\text{particle $i$	trajectory}\; \subset (0,\infty)\} \\ 
	 &=EZ_{t} P^{y}\{W_{s} 	\text{ does not hit $0$ by time }\;t \}\\
	 &=P^{y}\{W_{s} 	\text{ does not hit $0$ by time }\;t \}\\
	 & \longrightarrow 0 \quad \text{as} \; t \rightarrow \infty .
\end{align*}
\end{proof}

\subsection{Probability Generating Function of $N$}\label{ssec.gf}

Define
\begin{equation}
    \varphi(y,s)=\E^y[s^N]=\sum_{k=0}^{\infty}\Prob^{y}\{N=k \}s^{k}
    \label{def.phi.gen}
\end{equation}
to be the probability generating function of the random variable $N$
under the probability measure $\Prob^{y}$. Because the sequence $s^{n}$ is
multiplicative, Proposition
\ref{prop.varphi.ode} implies that for each complex number $s$ in the
disk $|s|<1$ the function $y\mapsto
\varphi (y,s)$ is $C^{2}$ and satisfies the differential equation 
\begin{equation}\label{eqn.gf.ode}
	\partial_{yy} \varphi (y,s)=\varphi (y,s)-\Psi (\varphi (y,s)) 
	\quad \text{for} \;\; y>0.
\end{equation}
Furthermore, $\varphi $ satisfies the boundary conditions
\begin{gather}\label{eqn.gf.bc}
	\varphi (0,s)=s \quad \text{and}\\
\notag 	\varphi (\infty ,s)=1.
\end{gather}

Because we are interested in the tail of the distribution, we will
find it useful to reformulate the boundary value problem for $\varphi$
as an equivalent problem for the generating function 
\begin{equation}
    H(y,s)=\sum_{k=1}^\infty \Prob^y(N\geq k)s^k.
    \label{def.H.gen}
\end{equation}

\begin{proposition}
    \label{prop.H}
    For each $s$ such that $|s|<1$, the function  $H(y,s)$ satisfies
    the differential equation
    \begin{equation}
        \partial_{yy} H(y,s)=\frac{s}{1-s}h\left( \frac{1-s}{s}H(y,s) \right).
        \label{eqn.H.ode}
    \end{equation}
    where $h (z)=2[\Psi (1-z)- (1-z)]$.
    In addition, $H(y,s)$ satisfies the following boundary conditions:
\begin{align}
        H (0,s)&=s, \label{eqn.H.kill} \\
        H (\infty ,s)&={0} ,\quad \text{and}
	\label{eqn.H.infty}\\
  	\lim_{s\to 1}H(y,s)&=1.
        \label{eqn.H.const}
\end{align}
\end{proposition}

\begin{proof}
   The generating functions $H$ and $\varphi$ are related by
    \begin{equation}
    \begin{aligned}
        \varphi(y,s)&= \sum_{k=0}^\infty \Prob^y(N=k)s^k\\
        &= \sum_{k=0}^\infty \Prob^y(N\geq k)s^k-\sum_{k=0}^\infty \Prob^y(N\geq k+1)s^k\\
        &= 1+H (y,s)\left(\frac{s-1}{s} \right).
    \end{aligned}
        \label{eqn.phi.G}
    \end{equation}
    Thus,  the differential equation \eqref{eqn.H.ode}
     follows directly from that for $\varphi$, as do boundary conditions
    \eqref{eqn.H.kill} and \eqref{eqn.H.infty}. Finally, the
    additional boundary condition \eqref{eqn.H.const} follows from the
    hypothesis that the branching Brownian motion is critical, as this
    makes $E^{y}N=1$, by Proposition~\ref{prop.exp.N}, and
    \begin{equation*}
        \lim_{s\to 1}H(y,s)=\sum_{k=1}^\infty \Prob^y(N\geq k)=\E^y[N].
    \end{equation*}
\end{proof}

\subsection{The Moranian Case}\label{ssec:killed-moranian}

Consider now the  Moranian case, where the number of
offspring is either $0$ or $2$, each with probability $\frac12$. In
this case the function $h$ in the differential equation
\eqref{eqn.H.ode}  reduces to $h (z)=z^{2}$, and so  \eqref{eqn.H.ode} becomes
\begin{equation}
    \partial_{yy} H(y,s)=\frac{1-s}{s}H^2(y,s).
    \label{eqn.H.ode.moranian}
\end{equation}
The  boundary conditions \eqref{eqn.H.kill} and \eqref{eqn.H.const}
uniquely determine the solution, which can be written  explicitly  as
\begin{equation}
    H(y,s)=s\left( \frac{1}{\sqrt{6}}y\sqrt{1-s}+1 \right)^{-2}
    \label{eqn.H.moranian}
\end{equation}

\begin{remark}
The function 
\begin{equation*}
    \wt H(y,s)=s\left( \frac{1}{\sqrt{6}}y\sqrt{1-s}-1 \right)^{-2}
\end{equation*}
also satisfies the boundary value problem, but since it has a pole at
$y=\sqrt{6/ (1-s)}$, it cannot be a probability generating function for
all $y>0$.
\end{remark}

Similarly, let
\begin{equation}
    u(y,s)=1-\varphi(y,1-s),
    \label{def.u.varphi}
\end{equation}
then the differential equation \eqref{eqn.gf.ode} and the boundary
conditions \eqref{eqn.gf.bc} become 
\begin{equation}
    \begin{aligned}
        \partial_{yy} u(y,s)&=u^2(y,s)\\
        u(0,s)&=s\\
        u(\infty,s)&=0,
    \end{aligned}
    \label{eqn.u.ode.moranian}
\end{equation}
which has the  solution
\begin{equation*}
    u(y,s)=\frac{6s}{(y\sqrt{s}+\sqrt{6})^2}.
\end{equation*}
so by \eqref{def.u.varphi}, the generating function $\varphi(y,s)$ is
\begin{equation}
    \varphi(y,s)=1-u(y,1-s)=1-\frac{6(1-s)}{(y\sqrt{1-s}+\sqrt{6})^2}.
    \label{eqn.varphi.moranian}
\end{equation}

The equation \eqref{eqn.varphi.moranian} completely determines the
distribution of $N$ under $\Prob^{y}$. In
Theorem~\ref{thm.killed.exact.moranian} below, we will use
\eqref{eqn.varphi.moranian} to provide explicit formulas for the
probabilities $\Prob^{y} (N\geq k)$. But before doing so, we will show
that \eqref{eqn.varphi.moranian} leads to the asymptotic formulas  
\eqref{eqn.v.asymp.general} and \eqref{eqn.p.asymp}.

\begin{theorem}
\label{thm.killed.moranian}
    In the Moranian case, 
    \begin{equation}
        \lim_{k\to\infty}k^{\frac32}\Prob^y(N\geq k)=
	\frac{y}{\sqrt{6\pi}}:=C_5 y,
        \label{eqn.v.asymp.moranian}
    \end{equation}
    and
    \begin{equation}
        \lim_{k\to\infty}k^{\frac52}\Prob^y(N=k)=
        \frac{3y}{2\sqrt{6\pi}}:=C_6 y.
        \label{eqn.varphi.asymp.moranian}
    \end{equation}
\end{theorem}

The proof will rely on the  following theorem of Flajolet and
Odlyzko \cite{FO90}.

\begin{theorem}[Corollary 2, \cite{FO90}] \label{cor.fo90} Assume that the
power series $A(z)=\sum_{n=1}^\infty a_nz^n$ defines an analytic
function in $|z|<1$ that has an analytic continuation to a
\emph{Pacman domain} 
\[
	D_{\circ b ,\delta }:= \{|z|<1+\delta \}\cap \{|\text{arg} (z-1)|>\beta  \}
\]
for some $\delta >0$ and $0\leq \beta <\pi /2$.
    If
    \begin{equation}\label{eqn.flajolet.hyp}
        A(z)\sim K(1-z)^{\alpha} \quad \text{as} \;\; z \rightarrow 1
	\;\; \text{in} \; D_{\alpha ,\delta},
    \end{equation}
    then as $n\to \infty$,
        \begin{equation}\label{eqn.flajolet.conclusion}
                a_n\sim \frac{K}{\Gamma(-\alpha)}n^{-\alpha-1}.
            \end{equation}
provided $\alpha\notin \left\{ 0,1,2,\cdots \right\}$.
\end{theorem}

\begin{proof}
[Proof of Theorem \ref{thm.killed.moranian}]
    The functions $H$ and $\varphi$ given by equations
    \eqref{eqn.H.moranian}  and  \eqref{eqn.varphi.moranian} have
    algebraic singularities at $s=1$, but for each $y$ have  unique analytic
continuations to the slit plane $\zz{C}\setminus \{ 1<s<\infty \}$. 
Expansion around $s=1$ yields
    
\begin{equation}\label{asym.H.moranian}
        \begin{aligned}
            H(y,s)&=s/{\left(\frac{1}{\sqrt{6}}y\sqrt{1-s}+1\right)^2}\\
            &=1-\frac{2y}{\sqrt{6}}(1-s)^{\frac12}+O(\a{1-s}). 
        \end{aligned}
    \end{equation}
    and
    \begin{equation}\label{asym.varphi.moranian}
        \begin{aligned}
            \varphi(y,s)&=1-(1-s)/
	    \left(1+\frac{2y}{\sqrt{6}}(1-s)^{\frac12}+\frac{y^2}{6}(1-s)\right)\\ 
            &=s+\frac{2y}{\sqrt{6}}(1-s)^{\frac32}+O(\a{1-s}^2).
        \end{aligned}
    \end{equation}
Thus, the hypotheses of Theorem \ref{cor.fo90} are satisfied, and
so the relations \eqref{eqn.v.asymp.moranian} and
\eqref{eqn.varphi.asymp.moranian} follow.
\end{proof}

Because $H$ is a simple algebraic function of the argument $s$, its
power series coefficients can be determined exactly. These provide an
explicit formula for the distribution of $N$ in the Moranian case.

\begin{theorem}
    In the Moranian case,
    \begin{equation}
        \Prob^y(N\geq k)=A_k+B_k,
        \label{eqn.killed.exact.moranian}
    \end{equation}
    where
    \begin{equation}
        A_k=k(k+c^2)\frac{c^{2k-2}}{(c^2-1)^{k+1}},
        \label{eqn.killed.exact.A}
    \end{equation}
    \begin{equation}
        B_k=2c\sum_{i+j=k,\newline i,j=0,1,\cdots} \frac{(2i-3)!!}{2^ii!}\frac{jc^{2j-2}}{(c^2-1)^{j+1}},
        \label{eqn.killed.exact.B}
    \end{equation}
    and
    \begin{equation}
        c=\frac{y}{\sqrt{6}}.
        \label{const.c}
    \end{equation}
    \label{thm.killed.exact.moranian}
\end{theorem}

\begin{proof}
The equation \eqref{eqn.H.moranian} for the generating function $H
(y,s)$ can be rewritten as
    \begin{equation*}
        H(y,s)= s(1+c\sqrt{1-s})^{-2}.
    \end{equation*}
Expanding around $s=0$ yields
    \begin{equation*}
        \begin{aligned}
            \frac{1}{\left[ 1+c\sqrt{1-s} \right]^2}&=\frac{\left[ 1-c\sqrt{1-s} \right]^2}{\left[ 1+c\sqrt{1-s} \right]^2\left[ 1-c\sqrt{1-s} \right]^2}\\
            &= \frac{1-2c\sqrt{1-s}+c^2(1-s)}{\left[ 1-c^2(1-s) \right]^2}\\
            &= \frac{1}{\left[ (1-c^2)+c^2s \right]^2}-\frac{2c\sqrt{1-s}}{\left[ (1-c^2)+c^2s \right]^2}+c^2\frac{(1-s)}{\left[ (1-c^2)+c^2s \right]^2}\\
            &= I+II+III.
        \end{aligned}
    \end{equation*}

    For $I$, we have
    \begin{equation*}
        \begin{aligned}
            I&= \left[ \frac{1}{(1-c^2)+c^2s} \right]^2\\
            &= \frac{1}{(1-c^2)^2}\left( \frac{1}{1-\frac{c^2}{c^2-1}s} \right)^2\\
            &= \frac{1}{(1-c^2)^2}\left[ \left(\sum_{i=0}^\infty \left(\frac{c^2}{c^2-1}\right)^is^i\right)\cdot \left(\sum_{j=0}^\infty \left(\frac{c^2}{c^2-1}\right)^js^j\right)\right]\\
            &= \frac{1}{(1-c^2)^2}\sum_{k=0}^\infty (k+1)\left(\frac{c^2}{c^2-1}\right)^k s^k\\
            &=\sum_{k=0}^\infty (k+1)\frac{c^{2k}}{(c^2-1)^{k+2}}s^k.
        \end{aligned}
    \end{equation*}

Now    let
    \begin{equation}
        \alpha_k=(k+1)\frac{c^{2k}}{(c^2-1)^{k+2}};
    \label{def.alpha.k}
    \end{equation}
then  for $III$, we have
    \begin{equation*}
        \begin{aligned}
            III&=c^2(1-s)\cdot I\\
            &=c^2\left(\sum_{k=0}^\infty \alpha_ks^k-\sum_{k=0}^\infty \alpha_ks^{k+1}\right)\\
            &=c^2\left( \sum_{k=0}^\infty\alpha_ks^k-\sum_{k=1}^\infty \alpha_{k-1}s^k \right)\\
            &=c^2\left( \alpha_0+\sum_{k=1}^\infty (\alpha_k-\alpha_{k-1})s^k \right)\\
            &\stackrel{\eqref{def.alpha.k}}{=}c^2\left( \frac{1}{(c^2-1)^2}+\sum_{k=1}^\infty \frac{c^{2k-2}}{(c^2-1)^{k+1}}\left[ \frac{c^2}{c^2-1}(k+1)-k \right]s^k\right)\\
            &=\sum_{k=0}^\infty \frac{c^{2k}}{(c^2-1)^{k+1}}\left[ \frac{c^2}{c^2-1}(k+1)-k \right]s^k.
        \end{aligned}
    \end{equation*}
Next, set
    \begin{equation}
        \beta_k=\frac{c^{2k}}{(c^2-1)^{k+1}}\left[ \frac{c^2}{c^2-1}(k+1)-k \right];
        \label{def.beta.k}
    \end{equation}
then we find that $A_{k+1}=\alpha_k+\beta_k$ by direct computation.

    It remains to show that
    \begin{equation*}
        II=\sum_{k=0}^\infty B_{k+1}s^k.
    \end{equation*}
     By Newton's binomial formula,
    \begin{equation}
        (1-s)^{1/2}=\sum_{i=0}^\infty {1/2 \choose i}(-s)^i.
        \label{eqn.frac.expan}
    \end{equation}
    Consequently,
    \begin{equation*}
        \begin{aligned}
            II&=-2c \left(\sum_{i=0}^\infty {1/2 \choose i}(-s)^i\right) \left(\sum_{j=0}^\infty \alpha_j s^j\right)\\ 
            &{=}-2c \left(\sum_{i=0}^\infty \frac{(-1)^{i-1}(2i-3)!!}{2^ii!}(-1)^is^i\right)\left(\sum_{j=0}^\infty \alpha_j s^j\right)\\ 
            &=2c\left( \sum_{i=0}^\infty \frac{(2i-3)!!}{2^ii!}s^i \right)\left(\sum_{j=0}^\infty \alpha_j s^j\right)\\
            &=2c\sum_{k=0}^\infty \left(\sum_{i+j=k,i,j=0,1\cdots} \frac{(2i-3)!!}{2^ii!}\alpha_j \right)s^k\\
            &\stackrel{\eqref{def.alpha.k}}{=}\sum_{k=0}^\infty B_{k+1}s^k.
        \end{aligned}
    \end{equation*}
\end{proof}

\subsection{Proof of
Theorem~\ref{theorem.killed-number.weak}}\label{ssec.killed-number.weak}

In the general case, where the offspring distribution is assumed only
to have mean $1$, variance $\sigma^{2}>0$, and finite third moment,
the behavior of the generating functions $\varphi (y,s)$ and $H (y,s)$
as $s \rightarrow 1-$ can be deduced from that in the Moranian case by
comparison arguments, as in section~\ref{sec:M-general}. This
derivation exploits the fact that for $s$ near $1$ the differential
equation for $H$ looks like that for the Moranian case, for which we
have exact solutions. The result, the details of whose proof we defer
to section~\ref{ssec:lemma.singular.behavior}, is as follows.

\begin{lemma}
\label{lemma.singular.behavior}
For each $y>0$, 
the generating functions $\varphi (y,s)$ and $H (y,s)$ satisfy
\begin{align}\label{eqn.varphi.H.singular}
    \varphi (y,s)-s &\sim \frac{2\sigma y}{\sqrt{6}} (1-s)^{3/2} \quad
	\text{and} \\
\label{eqn.H.singular}
	 H (y,s)-1&\sim \frac{-2\sigma y}{\sqrt{6}}(1-s)^{1/2}
\end{align}
as $s \uparrow 1$. 
\end{lemma}

Because the generating functions $\varphi (y,s)$ and $H (y,s)$ are
defined by power series with nonnegative coefficients, the singular
behavior of their derivatives can be deduced from
Lemma~\ref{lemma.singular.behavior}, by the following elementary fact.

\begin{lemma}\label{lemma.derivative-singular}
Let $A:[0,1]\rightarrow \zz{R}_{+}$ be an absolutely continuous, nonnegative,
increasing function whose derivative $A'$ is non-decreasing on
$(0,1)$.   If for some constants $C>0$
and $\alpha \in (0,1)$,
\begin{equation}\label{eqn.A-singular}
	A (1)-A (s)\sim C (1-s)^{\alpha} \quad \text{as} \;\; s \uparrow 1,
\end{equation}
then
\begin{equation}\label{eqn.derivative-singular}
	A' (s) \sim C\alpha (1-s)^{\alpha -1}  \quad \text{as} \;\; s \uparrow 1.
\end{equation}
\end{lemma}

\begin{proof}
Since $A$ is absolutely continuous,
\[
	A (s_{1})-A (s_{0})=\int_{s_{0}}^{s_{1}} A' (t) \,dt \quad \text{for all} \;
	0<s_{0}<s_{1}\leq 1. 
\]
Suppose that for some $\delta >0$ there were a sequence
$s_{n}\rightarrow 1-$ along which $A' (s_{n}) <C\alpha
(1-\delta)(1-s_{n})^{\alpha -1}$. Since $A'$ is non-decreasing, it
would then follow that $A' (s)<C\alpha
(1-\delta)(1-s_{n})^{\alpha -1}$ for all $s<s_{n}$, and so for any
$\varepsilon >0$,  
\[
	A (s_{n})-A (s_{n} (1-\varepsilon))\leq C\alpha
	(1-\delta)(1-s_{n})^{\alpha -1} \varepsilon  .
\]
But this would lead to a contradiction of the hypothesis
\eqref{eqn.A-singular} provided $\varepsilon$ is sufficiently small
relative to $\delta$. A similar argument shows that it is impossible
for $A' (s_{n}) >C\alpha (1+\delta)(1-s_{n})^{\alpha -1}$  along a
sequence $s_{n}\rightarrow 1-$.
\end{proof}

\begin{corollary}\label{corollary.H-at-one}
For each $y>0$, as $s \rightarrow 1-$,
\begin{equation}\label{eqn.H-at-one}
	\frac{d}{ds}H (y,s)=\sum_{k=1}^{\infty}k\Prob^{y} (N\geq k) s^{k} \sim
	\frac{\sigma y}{\sqrt{6}} (1-s)^{-1/2}.
\end{equation}
\end{corollary}

Theorem \ref{theorem.killed-number.weak} follows directly from
Corollary~\ref{corollary.H-at-one} and Karamata's Tauberian theorem
(cf. \cite{BGT87}, Corollary 1.7.3), which we now recall.

\begin{theorem}
Let $A (z)=\sum a_n z^n$ be a power series with nonnegative
coefficients $a_{n}$ and radius of convergence $1$. If,  for some
constants $C,\beta >0$,
\begin{equation}\label{eqn.karamata.hypothesis}
	A (s) \sim C/(1-s)^\beta \quad \text{as} \; s \uparrow 1,
\end{equation}
then as $n \rightarrow \infty$.
\begin{equation}\label{eqn.karamata.conclusion}
	\sum_{k=1}^n a_k \sim Cn^{\beta}/\Gamma(1+\beta).
\end{equation}
\end{theorem}

\qed

\subsection{Proof of Theorem \ref{thm.killed.asym.general}}\label{ssec.killed.general}

In this section we assume that the offspring distribution satisfies
the conditions enumerated in Theorem~\ref{thm.killed.asym.general}, in
particular, that the function $\kappa
(z)$ defined by \eqref{eqn.h.kappa} has no zeros $z$ such that
$0<|z|\leq 2$. We will once again make use of the Flajolet--Odlyzko
theorem (Theorem~\ref{cor.fo90}), which requires (a) that the function
defined by the power series in question should vary regularly as
functions of $s$ near the singularity $s=1$, and (b) that this
function has an analytic continuation to a Pacman
domain. Lemma~\ref{lemma.singular.behavior} implies that
$\varphi(y,s)$ and $H (y,s)$  vary regularly as $s \rightarrow 1-$
along the real axis from below. The following lemma ensures that they
have analytic continuations to a Pacman domain, and that the regular
variation persists in this region.

\begin{lemma}
\label{lemma.analytic-continuation} If the offspring
distribution satisfies the hypotheses of
Theorem~\ref{thm.killed.asym.general}, then for each $y>0$ the
generating functions $H (y,s)$ and $\varphi (y,s)$ have analytic
continuations $H (y,z)$ and $\varphi (y,z)$ to a slit disk
$\{|z|<1+\delta \}\setminus \{1\leq z <1+\delta \}$, and the relations
\eqref{eqn.varphi.H.singular} and \eqref{eqn.H.singular} hold as $s
\rightarrow 1$ in the slit domain.
\end{lemma}

\begin{proof}
In view of the relation \eqref{eqn.phi.G}, to prove that the function
$H (y,s)$ has an analytic continuation it suffices to show that
the probability generating function $\varphi (y,s)$ can be
analytically continued, or alternatively that the function 
\[
	u (y,z):= 1-\varphi (y,1-z)
\]
has an analytic continuation to a slit disk $\{|1-z|<1+\delta
\}\setminus {\{1 \leq 1-z< 1+\delta \}}$. Recall
(Proposition~\ref{prop.varphi.ode}) that for all \emph{real} $z$ in
the interval $|1-z|<1$ the function $u (y,z)$ satisfies the boundary
value problem
\begin{equation}\label{eqn.u.bvp}
\begin{aligned}
	 \partial_{yy} u(y,z)&=h(u(y,z)),  \\
	 u(0,z)&=z,     \\
	u (\infty ,z)&=0.	
\end{aligned}
\end{equation}
Integration of this differential equation , as in section 3 of
\cite{SF79}, leads to the equation  
\begin{equation}\label{eqn.integral-equation}
	\int_{u (x,z)}^{z} \frac{dy}{\sqrt{\kappa (y)}}=x
	\quad \text{where} \quad 
	\kappa (y):=2\int_{0}^{y} h (y') \,dy'.
\end{equation}
(The upper limit of integration is $z$ because $u (0,z)=z$.) 

It is easily checked that in any region of the $z-$plane where the
equation \eqref{eqn.integral-equation} has a solution $u (x,z)$, the
solution will satisfy the boundary value problem
\eqref{eqn.u.bvp}. Thus, to prove the first assertion of the lemma it
will suffice to show that the integral equation
\eqref{eqn.integral-equation} implicitly defines $u (x,z)$ as an
analytic function of $z$ for $z$ in a slit disk.

Define a function of two complex variables $z,w$ by
\[
	G (w,z):=\int_{w}^{z} \frac{dy}{\sqrt{\kappa (y)}}.
\]
This function is analytic in $z$ and $w$ in any domain $\hat{D} \subset \zz{C}^{2}$
such that there exists a simply connected domain $D\subset \zz{C}$ in
which $1/\sqrt{\kappa}$ is analytic and such that for any $z,w\in
\zz{C}$ there is a  path from $z$ to $w$ in $D$.  Moreover,
\[
	\frac{\partial G}{\partial z}= 1/\sqrt{\kappa (z)} \quad 
	\text{and} \quad \frac{\partial G}{\partial w}= 1/\sqrt{\kappa
	(w)}. 
\]
Therefore, by the complex implicit function theorem, the equation $G
(w (z),z)=x$ defines $w (z)$ as an analytic function of $z$ in a
neighborhood of any solution $G (w_{0},z_{0})=x$ where $1/\sqrt{\kappa
(w_{0})}\not =0$. 

The function $\varphi (x,z)$ is analytic in the unit disk and, since
it is a probability generating function, satisfies $|\varphi
|<1$. Consequently, the function $u (x,z)=1-\varphi (x,1-z)$ is
analytic in the disk $D_{1}:=\{|1-z|<1 \}$ and satisfies $u (x,z)\in
D_{1}$ for all $z\in D_{1}$. By hypothesis, the
functions $h (z)$ and $\kappa (z)$ have analytic continuations to a
disk of radius $2+\delta >2$ centered at $0$, and the only zero of
$\kappa (z)$ in this disk is at $z=0$. Therefore, the functional
equation 
\begin{equation}\label{eqn.functionalEqn}
	G (u (x,z),z)=x
\end{equation}
holds for all $z\in D_{1}$. 

We claim that for any point $\xi \in \partial D_{1}$ \emph{except}
$\xi =0$, the function $u (x,z)$ converges as $z \rightarrow \xi$ to a
value $u (x,\xi)$ such that $u (x,\xi)\in D_{1}$. This is clearly
equivalent to the assertion that for any point $e^{i\theta}\not =1$
of the unit circle, the function $\varphi (x,z)$ converges as $z
\rightarrow e^{i\theta}$ to a value $\varphi (x,e^{i\theta})$ of
absolute value less than $1$. To see that this is so, recall that
$\varphi (x,z)$ is the probability generating function of the random
variable $N$ under $P^{x}$. It is easily seen (for instance, using the
discrete Brownian snake construction) that for any $x>0$,
\[
	P^{x}\{N=k \}>0 \quad \text{for every} \;\; k=0,1,2,\dots .
\]
But this implies that $|E^{x}e^{i\theta N}|$ is less than $1$ for
every $\theta \in [-\pi ,\pi ]\setminus \{0 \}$.

It now follows that the functional equation \eqref{eqn.functionalEqn}
extends by continuity to all $z\in \partial D_{1}$ except $z=0$, and
that at any such boundary point, $u (x,z)\not =0$. Hence, the function
$1/\sqrt{\kappa (w)}$  is analytic and nonzero  in a neighborhood of
$w=u (x,z)$. This implies that $u (x,z)$ has an analytic continuation
to a neighborhood of every $z\in \partial D_{1}$ except the  singular
point $z=0$.

Next, we must prove that $u (x,z)$ has an analytic continuation
to a {slit} neighborhood of $z=0$. Since $\varphi (x,1)=1$, the
function $u (x,z)$ converges to $0$ as $z \rightarrow  0$ in the disk
$\{|1-z|<1 \}$; thus, the function $\kappa (u)$ approaches $0$.
Our assumptions on the offspring distribution imply that 
\begin{align*}
	h (z)&=\sigma^{2}z^{2}+\sum_{j=3}^{\infty } h_{j}z^{j} \quad \text{and hence}\\
	\kappa (z)&= \frac{2}{3}\sigma^{2}z^{3} +\sum_{j=4}^{\infty }
	k_{j}z^{j};
\end{align*}
consequently,
\[
    \frac{1}{\sqrt{\kappa (z)}}= Kz^{-3/2} (1+R (z))
\]
where ${K=\sqrt{3/2\sigma^{2}}}$ and $R (z)$ is an analytic function in
some neighborhood of $z=0$ and satisfies  $R (0)=0$. Now the integral
equation \eqref{eqn.integral-equation} and the implicit function
theorem imply that, for fixed $x>0$, the function $u=u (x,z)$
satisfies the differential equation 
\[
	\frac{du}{dz}=\frac{\sqrt{\kappa (u)}}{\sqrt{\kappa
	(z)}}=\frac{u^{3/2}}{z^{3/2}}\frac{(1+R (z))}{(1+R
	(u))}
\]
for all $z$ in a domain  $\{|1-z|<1 \}\cap \{|z|<\delta
\}$. Using the analyticity of $R$, we conclude that
\[
	u^{-1/2}\left(1+\sum_{k=1}^{\infty}b_{k}u^{k} \right)
	=	z^{-1/2}\left(1+\sum_{k=1}^{\infty}b_{k}z^{k} \right) +C,
\]
where $C$ is a constant of integration. Squaring both sides exhibits
$u$ as a meromorphic function of $\sqrt{z}$. This shows that $u$ has an
analytic continuation to a slit disk $\{|z|<\delta \}\setminus
(-\delta ,0]$, and so it follows, by the relations $u (x,z)=1-\varphi (x,1-z)$
and equation  \eqref{eqn.phi.G}, that $\varphi (x,z)$ and $H (x,z)$
admit analytic continuations to a slit disk $\{|1-z|<\delta '\}\setminus
[1,1+\delta ']$.

Finally, we must prove that the relations
\eqref{eqn.varphi.H.singular} and \eqref{eqn.H.singular} hold as
$z\rightarrow 1$ in the extended domain of definition of the functions 
$\varphi (x,z)$ and $H (x,z)$. But the analytic continuation argument
above shows that, in a slit disk centered at $z=1$, the functions
$\varphi (x,z)-z$ and $H (x,z)-1$ are meromorphic functions of
$\sqrt{1-z}$, and hence have Puiseux expansions in powers of
$(1-z)^{1/2}$. Since \eqref{eqn.varphi.H.singular} and
\eqref{eqn.H.singular} hold as $s \uparrow 1$, it follows that they
persist in the slit disk.

\end{proof}

\begin{proof}
[Proof of Theorem~\ref{thm.killed.asym.general}]
Lemma \ref{lemma.analytic-continuation} implies that the generating
functions $\varphi (y,s)$ and $H (y,s)$ meet  the requirements of   the
Flajolet-Odlyzko theorem (Corollary~\ref{cor.fo90}).  Therefore,
Theorem~\ref{thm.killed.asym.general} follows from
relation~\eqref{eqn.flajolet.conclusion}  (since
$\Gamma(-\frac12)=-2\sqrt{\pi}$ and
$\Gamma(-\frac32)=\frac{4}{3}\sqrt{\pi}$). 
\end{proof}

\subsection{Proof of  Lemma
\ref{lemma.singular.behavior}}\label{ssec:lemma.singular.behavior} 

The strategy is similar to that of
section~\ref{sec:M-general}. As 
$s\rightarrow 1$,
\begin{equation*}
    \begin{aligned}
        u(y,1-s)&\rightarrow 0\\
        \frac{1-s}{s}H(y,s)&\rightarrow 0,
    \end{aligned}
\end{equation*}
and so for $s$ near $1$ the 
differential equations \eqref{eqn.u.bvp} and \eqref{eqn.H.ode} for $u
(y,1-s)$ and $H (y,s)$ have forcing terms that are nearly
quadratic. Taylor expansion of $h$  shows that for any $\delta >0$
there exists $\varepsilon  >0$ such that for $1-\varepsilon<s<1  $,
\begin{equation*}
    \begin{aligned}
        a_-u^2(y,1-s)\leq &h(u(y,1-s))\leq a_+u^2(y,1-s) \quad \text{and}\\
        a_-\left[ \frac{1-s}{s}H(y,s) \right]^2\leq & h\left(
	\frac{1-s}{s}H(y,s) \right)\leq a_+ \left[ \frac{1-s}{s}H(y,s)
	\right]^2, 
    \end{aligned}
\end{equation*}
where $a_\pm$ are defined in \eqref{def.a+-}. Let $u_\pm(y,1-s)$ and $H_{\pm}(y,s)$ satisfy the
boundary value problems
\begin{equation*}
    \left\{
    \begin{aligned}
        \partial_{yy} u_\pm(y,1-s)&=a_\pm u_\pm^2(y,1-s)\\
        u_\pm(0,s)&=s\\
        u_\pm(\infty,s)&=0,
    \end{aligned}
    \right.
    \left\{
    \begin{aligned}
        \partial_{yy} H_\pm(y,s)&=a_\pm\frac{1-s}{s} H_\pm^2(y,s)\\
        H_\pm(0,s)&=s\\
        H_\pm(\infty,s)&=0,
    \end{aligned}
    \right.
\end{equation*}
and set $\varphi_\pm(y,s)=1-u_\pm(y,1-s)$. By the same argument as in
Corollary \ref{cor.pinching},
\begin{equation}
    \label{eqn.pinching.killed}
    \begin{aligned}
        u_+(y,1-s)\leq &u(y,1-s)\leq u_-(y,1-s)\\
        \varphi_-(y,s)\leq &\varphi(y,s)\leq \varphi_+(y,s)\\
        H_+(y,s)\leq &H(y,s)\leq H_-(y,s).
    \end{aligned}
\end{equation}

Define re-scaled versions
\begin{equation}
    \begin{aligned}
        \wh u_\pm(y,1-s)&=u_\pm\left(\frac{y}{\sqrt{a_\pm}},1-s\right)\\
        \wh H_\pm(y,s)&=H_\pm\left(\frac{y}{\sqrt{a_\pm}},s\right);
    \end{aligned}
    \label{eqn.u.H.scaled}
\end{equation}
these satisfy the boundary value problems
\begin{equation*}
    \left\{
    \begin{aligned}
        \partial_{yy} \wh u_\pm(y,1-s)&=\wh u_\pm^2(y,1-s)\\
        \wh u_\pm(0,s)&=s\\
        \wh u_\pm(\infty,s)&=0,
    \end{aligned}
    \right.
    \left\{
    \begin{aligned}
        \partial_{yy} \wh H_\pm(y,s)&=\frac{1-s}{s}\wh H_\pm^2(y,s)\\
        \wh H_\pm(0,s)&=s\\
        \wh H_\pm(\infty,s)&=0,
    \end{aligned}
    \right.
\end{equation*}
which are the same as in the Moranian case. Hence, $\wh
\varphi_\pm(y,s)=1-\wh u_\pm(y,1-s)$ and $\wh H_\pm(y,s)$ have the
same asymptotics as \eqref{asym.varphi.moranian} and
\eqref{asym.H.moranian}: 
\begin{equation*}
    \begin{aligned}
        \wh \varphi_\pm(y,s)=s+\frac{2y}{\sqrt{6}}(1-s)^{\frac32}+O\left(\a{1-s}^2\right)\\
        \wh H_\pm(y,s)=1-\frac{2y}{\sqrt{6}}(1-s)^{\frac12}+O\left( \a{1-s} \right)
    \end{aligned}
\end{equation*}
as $s\rightarrow 1$. Hence, applying \eqref{eqn.u.H.scaled} yields
\begin{equation*}
    \begin{aligned}
        \varphi_\pm(y,s)=s+\frac{2y\sqrt{a_\pm}}{\sqrt{6}}(1-s)^{\frac32}+O\left(\a{1-s}^2\right)\\
        H_\pm(y,s)=1-\frac{2y\sqrt{a_\pm}}{\sqrt{6}}(1-s)^{\frac12}+O\left( \a{1-s} \right).
    \end{aligned}
\end{equation*}

Finally, as $\delta\to 0$ we have $a_\pm\to \sigma^2$, and so \eqref{eqn.pinching.killed} implies
\begin{equation*}
    \begin{aligned}
        \varphi(y,s)-s&\sim \frac{2\sigma y}{\sqrt{6}}(1-s)^{\frac32}\\
        H(y,s)-1&\sim \frac{-2\sigma y}{\sqrt{6}}(1-s)^{\frac12}.
    \end{aligned}
\end{equation*}
as $s\to 1$.

\qed


\begin{thebibliography}{99}
    \bibitem[AS72]{AS72}, M. Abramowitz and I. Stegun, \emph{Handbook of Mathematical Functions with Formulas, Graphs and Mathematical Tables}, New York, Dover Publications, ISBN 978-0-486-61272-0, 1972.
    \bibitem[A]{A} D. Aldous, Power laws and killed branching random walks, Available at \emph{http://www.stat.berkeley.edu/\~aldous/Research/OP/brw.html}
    \bibitem[AB11]{AB11} L. Addario-Beryy and N. Broutin, Total progeny in killed branching random walk, \emph{Probab. Theory Relat. Fields}, Vol 151, No. 1-2, pp. 265--295, 2011.
    \bibitem[AHZ13]{AHZ13} E. Aidekon, Y. Hu and O. Zindy, The precise tail behavior of the total progeny of a killed branching random walk, \emph{Ann. Probab.}, Vol 41, No. 6, pp. 3786--3878. 2013.
    \bibitem[AN72]{AN72} K. Athreya and P. Ney, \emph{Branching processes}, Springer-Verlag, New York, 1972. Die Grundlehren der mathematischen Wissenschaften, Band 196.
    \bibitem[BBS13]{BBS13} J. Berestycki, N. Berestycki and
    J. Schweinsberg, The genealogy of branching Brownian motion with
    absorption, \emph{Ann. Probab.}, Vol. 41, No. 2, 2013.
\bibitem[BBHM15]{BBHM15} J. Berestycki, E. Brunet, S. Harris, and
P. Milos, Branching Brownian motion with absorption and the all-time
minimum of branching Brownian motion with drift. Preprint, available
at {\tt http://arxiv.org/abs/1506.01429}.
\bibitem[BGT87]{BGT87}
N.~H. Bingham, C.~M. Goldie, and J.~L. Teugels.
\newblock {\em Regular variation}, volume~27 of {\em Encyclopedia of
  Mathematics and its Applications}.
\newblock Cambridge University Press, Cambridge, 1987.
    \bibitem[B78]{B78} M. Bramson, Maximal displacement of branching Brownian motion, \emph{Comm. Pure. Appl. Math.}, Vol 31, No. 5, pp. 531--581, 1978.
    \bibitem[CG76]{CG76} K. Crump and J. Gillespie, The dispersion of a neutural allele considered as a branching process. \emph{J. Appl. Probability}, Vol. 13, No. 2, pp. 208--218, 1976.
    \bibitem[FO90]{FO90} P. Flajolet and A. Odlyzko, Singularity analysis of generating functions, \emph{SIAM J. Disc. Math.}, 1990.
    \bibitem[K78]{K78} H. Kesten, Branching brownian motion with absorbption, \emph{Stochastic Processes Appl.}, Vol. 7, No. 1, pp. 9--47, 1978.
    \bibitem[K84]{K84} N. Koblitz, \emph{Introduction to elliptic curves and modular forms}, Springer-Verlag, 1984.
    \bibitem[KPP37]{KPP37} A. Kolmogorov, I. Petrovskii and N. Piscounov, A study of the diffusion equation with increase in the amount of substance, and its application to a biological problem, \emph{Selected Works of A.N. Kolmogorov I}, pp. 248--270, Kluwer, 1991, Translated by V.M. Volosov from \emph{Bull. Moscow Univ. Math. Mech.}, Vol. 1, pp. 1--25, 1937.
    \bibitem[KK03]{KK03} A. Korostelev and O. Korosteleva, Limit theorem for the spread of branching diffusion with stabilizing drift, \emph{Stochastic Anal. Appl.}, Vol. 21, No. 3, 611--641, 2003.
    \bibitem[KK04]{KK04} A. Korostelev and O. Korosteleva, Spread of branching diffusion: sharp asymptotics, \emph{Stoch. Stoch. Rep.}, Vol. 76, No. 6, 509--516, 2004.
    \bibitem[K04]{K04} O. Korosteleva, The spread of branching diffusion under stabilizing drift, \emph{Stoch. Stoch. Rep.}, Vol. 76, No. 3, 179--189, 2004.
    \bibitem[L09]{L09} S. Lalley, Spatial epidemics: critical behavior in one dimension, \emph{Probab. Theorey Related Fields}, Vol. 144, No. 3--4, pp. 429--469, 2009.
    \bibitem[LS87]{LS87} S. Lalley and T. Sellke, A conditional limit theorem for the frontier of a branching Brownian motion. \emph{Ann. Probab.}, Vol. 15, No. 3, pp. 1052--1061, 1987.
    \bibitem[LS88]{LS88} S. Lalley and T. Sellke, Traveling waves in inhomogeneous branching {B}rownian motions. {I}, \emph{Ann. Probab.}, Vol. 16, No. 3, 1051--1062, 1988.
    \bibitem[LS89]{LS89} S. Lalley and T. Sellke, Traveling waves in inhomogeneous branching {B}rownian motions. {II}, \emph{Ann. Probab.}, Vol. 17, No. 1, 116--127, 1989.
    \bibitem[LS92]{LS92} S. Lalley and T. Sellke, Limit theorems for the frontier of a one-dimensional branching diffusion, \emph{Ann. Probab.}, Vol. 20, No. 3, 1310--1340, 1992.
    \bibitem[L90-1]{L90-1} T. Lee, Conditioned limit theorems of stopped critical branching bessel processes, \emph{The Annals of Probability}, Vol. 18, No. 1, 272--289, 1990.
    \bibitem[L90-2]{L90-2} T. Lee, Some Limit Theorems for Critical Branching Bessel Processes, and Related Semilinear Differential Equations, \emph{Probab. Th. Rel. Fields}, Vol. 84, No. 4, 505--520, 1990.
    \bibitem[L99]{L99} J. LeGall, \emph{Spatial branching processes, random snakes and partial differential equations}, Lectures in Mathematics ETH Z\"urich, Birkh\"auser Verlag, Basel, 1999.
    \bibitem[M13]{M13} P. Maillard, The number of absorbed individuals in branching Brownian Motion with a barrier, \emph{Annales de l'Institut Henri Poincar\'e - Probabilit\'es et Statistiques}, Vol. 49, No. 2, 428--455, 2013.
    \bibitem[M69]{M69} H. McKean, \emph{Stochastic integrals}, Probability and Mathematical Statistics, No. 5, Academic Press, New York-London, 1969.
    \bibitem[M75]{M75} H. McKean, Application of Brownian Motion to the equation of Kolmogorov-Petrovskii-Piskunov, \emph{Comm. Pure. Appl. Math}, Vol. 28, No. 3, pp. 323--331, 1975
    \bibitem[MM99]{MM99} H. McKean and V. Moll, \emph{Elliptic
    curves}, Cambridge University Press, 1999.
   \bibitem[N88]{N88} J. Neveu, Multiplicative martingales for
   spatial branching processes, \emph{Seminar on Stochastic Processes,
   1987}, 223--242, 1988.
    \bibitem[S76]{S76} S. Sawyer, Branching diffusion processes in population genetics, \emph{Adv. Appl. Prob.}, Vol. 8, No. 4, pp. 659--689, 1976.
    \bibitem[SF79]{SF79} S. Sawyer and J. Fleischman, Maximum geographic range of a mutant alle considered as a subtype of branching random field, \emph{Proc. Natl. Acad. Sci.}, Vol. 76, No. 2, pp 872--875, 1979.
\end{thebibliography}
\end{document}